\documentclass[11pt]{article}
\usepackage{amsfonts}
\usepackage{amsmath}
\usepackage{amsthm}
\usepackage{amscd}
\usepackage{amsbsy}
\usepackage{graphicx}
\usepackage{indentfirst, latexsym, bm,amssymb}
\usepackage{bbding}

\usepackage[pagewise]{lineno}
\usepackage[colorlinks=true]{hyperref}
\advance\textwidth by +1.0in \advance\textheight by +1.0in
\advance\oddsidemargin by -0.5in \advance\evensidemargin by -1.0in
\advance\topmargin by -0.5in
\newtheorem {theorem} {Theorem}
\newtheorem {proposition} [theorem]{Proposition}
\newtheorem {corollary} [theorem]{Corollary}
\newtheorem {lemma}  [theorem]{Lemma}

\newtheorem {notation} [theorem]{Notation}

\usepackage{CJK}
\newtheorem{Defi}[theorem]{Definition}

\title{\large\bf On the Dirichlet Problem at Infinity and Poisson Boundary for Certain Manifolds without Conjugate Points}

\author{
\and Fei Liu
\thanks{College of Mathematics and Systems Science, Shandong University of Science and Technology, Qingdao, 266590, P.R. China.
e-mail: liufei$@$math.pku.edu.cn }
\and Yinghan Zhang
\thanks{College of Mathematics and Systems Science, Shandong University of Science and Technology, Qingdao, 266590, P.R. China.
e-mail: zhangyinghan007@126.com }
\and Dedicated to Professor Xiang Zhang on the occasion of his 60th birthday
}

\date{\today}

\begin{document}
\maketitle

\begin{abstract}
In this paper, we investigate the problem of the existence of the bounded harmonic functions on a simply connected Riemannian manifold $\widetilde{M}$ without conjugate points, which can be compactified via the ideal boundary $\widetilde{M}(\infty)$. Let $\widetilde{M}$ be a uniform visibility manifold which satisfy the Axiom $2$,
or a rank $1$ manifold without focal points, suppose that $\Gamma$ is a cocompact discrete subgroup of $Iso(\widetilde{M})$, we show that for a given continuous function on $\widetilde{M}(\infty)$, there exists a harmonic extension to $\widetilde{M}$. And furthermore, when $\widetilde{M}$ is a rank $1$ manifold without focal points, the Brownian motion defines a family of harmonic measures $\nu_{\ast}$ on $\widetilde{M}(\infty)$, we show that $(\widetilde{M}(\infty),\nu_{\ast})$ is isomorphic to the Poisson boundary of $\Gamma$.
 \\


\end{abstract}


\section{\bf Introduction and Main Results}\label{intro}
\setcounter{section}{1}
\setcounter{equation}{0}\setcounter{theorem}{0}
The problem of the existence of the bounded harmonic functions on a simply connected manifold is
an important and fascinating research topic in modern mathematics.
It is well known that every bounded harmonic function on $\mathbb{R}^{n}$ is constant.
One of the basic ingredient of the proof for this fact is that the growth rate of the volume of balls with respect to the radius is polynomial.
On the other hand, there are lots of simply connected Riemannian manifolds whose volume are exponentially growth.
In this paper, we will concern the existence of the bounded harmonic functions on such manifolds.

Let $(M,g)$ be a compact Riemannian manifold, and $(\widetilde{M},\widetilde{g})$ be its universal covering manifold.
If we denote by $\Gamma\triangleq\pi_1(M)$,
then $\Gamma$ can be viewed as a discrete subgroup of the isometry group of $\widetilde{M}$, and $M=\widetilde{M}/\Gamma$.
To simplify the notations, we shall denote $(M,g)$ and $(\widetilde{M},\widetilde{g})$ simply by $M$ and $\widetilde{M}$, respectively.
Furthermore, we use $T^1 M$ and $T^1\widetilde{M}$ to denote the unit tangent bundles of $M$ and $\widetilde{M}$, respectively.
For any $v\in T^1 M$ or $v\in T^1 \widetilde{M}$, we use $c_{v}$ to denote the unique unit speed geodesic with the initial condition $c'_{v}(0)=v$.
We use $$v\mapsto \phi_{t}(v)=c'_{v}(t)$$ to denote the \emph{geodesic flow} on the unit tangent bundle $T^1 M$.

The volume growth rate of $\widetilde{M}$ has a profound connection with the dynamics of the geodesic flow on $T^1 M$.
The exponential growth rate of the volume of the balls in $\widetilde{M}$ is called $\mathbf{volume~entropy}$.
It is now a classical result that proved by A. Manning states that
the volume entropy of $\widetilde{M}$ equals to the topological entropy of the geodesic flow on $M$, when the manifold is non-positively curved,
see \cite{M} for more details. Later, this result is extended by A. Freire and R. Ma\~n\'e to the manifolds without conjugate points in \cite{FM}.
In this paper, we will consider the existence of the bounded harmonic functions on certain manifolds without conjugate points.

Let $c$ be a geodesic on $M$. A pair of distinct points $p=c(t_{1})$ and $q=c(t_{2})$ are called $\mathbf{focal}$ if there is a Jacobi field $J$ along $c$ such that $J(t_{1})=0$, $J'(t_{1})\neq 0$ and $\frac{d}{dt}\| J(t)\|^{2}\mid_{t=t_{2}}=0$; $p=c(t_{1})$ and $q=c(t_{2})$ are called $\mathbf{conjugate}$ if there is a non-identically-zero Jacobi field $J$ along $c$ such that $J(t_{1})=0=J(t_{2})$.

A typical example is the standard two dimensional sphere $S^{2}$, each pair of the antipodal points are conjugate, and each pair of points (N, p) are focal where
$N$ is the north pole and $p$ is any points in the equator.

\begin{Defi}  \label{def0}
$M$ is called a manifold \emph{without focal points / without conjugate points} if there is no focal points / conjugate points on any geodesic in $M$.
The universal cover $\widetilde{M}$ of $M$ is called a manifold \emph{without focal points / without conjugate points} if $M$ does.
\end{Defi}

It's easy to see that the fact ``no focal points" implies ``no conjugate points". It is well known that non-positively curved manifolds having no focal points. Manifolds without focal points / without conjugate points can be regarded as a non-trivial generalization of the manifolds with non-positive curvature, since there are manifolds without focal points / without conjugate points whose curvature are not everywhere non-positive (cf. \cite{Gu}).

One can generalize the classical Cartan-Hadamard theorem from the manifolds with non-positive curvature to the manifolds without conjugate points,
thus $\widetilde{M}$ is diffeomorphic to $\mathbb{R}^{n}$ where $n=\mathrm{dim} \widetilde{M}$ (cf. \cite{OS2}).
We shall endow $\widetilde{M}$ with a boundary, called the ideal boundary, denoted by $\widetilde{M}(\infty)$,
such that $\widetilde{M} \cup\widetilde{M}(\infty)$ is compact under the so called cone topology.

As we all known that the volume entropy of a pinched simply connected negatively curved manifold is positive, i.e., the volume growth is exponential.
However, when we turn our attention to the manifolds with non-positive curvature or, more generally,
the manifolds without focal points or the manifolds without conjugate points, the situation becomes significantly more complicated.
In all the three cases, there exist examples where the volume entropy is either zero or positive.

Therefore, a natural question is: which non-positively curved manifolds (or manifolds without focal points, or manifolds without conjugate points)
have positive exponential volume growth rates?

For the non-positively curved manifolds or the manifolds without focal points, rank is a concept of fundamental importance (cf. \cite{BBE}).
We note that the geodesic flow on such manifolds admits a unique measure of maximal entropy (cf. \cite{Kn, LWW}),
thus by the classical results of \cite{M, FM} mentioned above,
we know that the volume entropies are positive, i.e., the volume of the balls in these manifolds have positive exponential growth rates.

For the manifolds without conjugate points,
unfortunately, there is no uniform definition of rank.
We note that G. Knieper and N. Peyerimhoff in \cite{KP} and L. Rifford and R. Ruggiero in \cite{RR} both give definition of the
rank separately, but their definitions are different.
Recently, E. Mamani and R. Ruggiero have investigated the existence and uniqueness of maximal entropy measures for geodesic flows on manifolds without conjugate points,
under some additional conditions (cf. \cite{MR}). While investigating the Hopf-Tsuji-Sullivan (HTS) dichotomy theory,
the first author of this paper and his collaborators in \cite{LLW} (see also \cite{Wu1, Wu3}), study the geometry and dynamics for a class of non-compact manifolds without conjugate points systematically.
By \cite{MR}, it can be seen that a compact manifold satisfying the conditions given in our paper  \cite{LLW} admits a measure of maximal entropy,
thus the exponential growth rates of volume of such manifolds without conjugate points are positive.

In this paper, we shall investigate the existence of the bounded harmonic functions on such simply connected Riemannian manifold, i.e.,
simply connected rank $1$ manifolds without focal points or simply connected manifolds without conjugate points, with some additional conditions as described in \cite{LLW}.
More precisely we will consider the so called Dirichlet problem at infinity (or the asymptotic Dirichlet problem) on such manifolds.

$\mathbf{Dirichlet~problem~at~infinity}$:
Given a real continuous function $f$ on the ideal boundary $\widetilde{M}(\infty)$, does there exist a harmonic function $h$ on $\widetilde{M}$
whose continuous extension to the ideal boundary coincides precisely with $f$?

For the pinched negatively curved manifolds, the Dirichlet problem at infinity is valid for any continuous function $f$ on the ideal boundary (cf. \cite{An, S}).
Then W. Ballmann in \cite{B} solved this problem on the rank $1$ manifolds with non-positive curvature.

The following Theorem is one of the main results of this paper.

\begin{theorem}[The solution of the Dirichlet problem at infinity beyond non-positive curvature]\label{Dirichlet}
Let $\widetilde{M}$ be a simply connected uniform visibility manifold without conjugate points which satisfies the Axiom $2$,
or a rank $1$ simply connected manifold without focal points,
and let $\Gamma$ be a cocompact discrete subgroup of the isometry group $Iso(\widetilde{M})$, i.e., $\widetilde{M}/\Gamma$ is compact,
then the Dirichlet problem at infinity can be solved for any continuous function on the ideal boundary $\widetilde{M}(\infty)$.
\end{theorem}

All the notions appeared in Theorem~\ref{Dirichlet} will be described in details in Section~\ref{geometry}.

As mentioned above, the manifolds appeared in Theorem~\ref{Dirichlet} have positive exponential volume growth rates.
We note that recently, in a remarkable paper \cite{Po}, P. Polymerakis studies the existence of positive harmonic functions on such manifolds,
and get some deep results.

Theorem~\ref{Dirichlet} indicates that $\widetilde{M}$ has a large class of bounded harmonic functions. Indeed,
by Theorem~\ref{Dirichlet}, for any continuous function $f:\widetilde{M}(\infty)\rightarrow \mathbb{R}$,
the harmonic function $h:\widetilde{M}\rightarrow \mathbb{R}$ is defined as
\begin{equation}\label{eq81}
p \longmapsto h(p)\triangleq \int_{\widetilde{M}(\infty)}f(\xi)d\nu_{p}(\xi),
\end{equation}
where $\nu_{p}$ is the hitting probability measure on the ideal boundary corresponding to the Brownian motion starting at $p$,
we denote by $\nu_{\ast}=\big\{\nu_{p}\big\}_{p\in \widetilde{M}}$.
For more details, please see Section~\ref{BMS}.

Let $\mathcal{H}^{\infty}(\widetilde{M})$ be the space of bounded harmonic functions on $\widetilde{M}$.
The classical Poisson representation formula shows that $\mathcal{H}^{\infty}(\mathbb{H}^{n})$ is isomorphic to $L^{\infty}(\mathbb{S}^{n-1})$
for the $n$-dimensional hyperbolic space $\mathbb{H}^{n}$.
It is widely known that the hyperbolic space $\mathbb{H}^{n}$ is one of the rank $1$ symmetric spaces of non-compact type. In \cite{Fu2},
H. Furstenberg extended this kind of representation to all the symmetric spaces of non-compact type.
Then M. Anderson and R. Schoen generalized this result to the pinched negatively curved manifolds in \cite{AS}.
Later, W. Ballmann and F. Ledrappier obtained such a representation formula for the rank 1 manifolds with non-positive curvature in \cite{BL}.

The following Theorem states that for the manifolds stated in Theorem~\ref{Dirichlet}, every bounded harmonic function on $\widetilde{M}$
can be represented as in Formula \eqref{eq81}.

\begin{theorem}\label{Poisson}
Let $\widetilde{M}$ be a rank $1$ simply connected manifold without focal points,
and $\Gamma$ be a cocompact discrete subgroup of the isometry group $Iso(\widetilde{M})$,
then the Formula \eqref{eq81} gives an isomorphism from $\mathcal{H}^{\infty}(\widetilde{M})$ to $L^{\infty}(\widetilde{M}(\infty),\nu_{\ast})$.
Furthermore, for each $h\in\mathcal{H}^{\infty}(\widetilde{M})$, the corresponding $f\in L^{\infty}(\widetilde{M}(\infty),\nu_{\ast})$
under this isomorphism satisfies
$$
f\big(\lim_{t\rightarrow +\infty}\omega(t)\big)=\lim_{t\rightarrow +\infty}h\big(\omega(t)\big),
$$
for almost every path $\big\{\omega(t)\big\}_{t\in \mathbb{R}^{+}}$ of the Brownian motion defined on $\widetilde{M}$.
\end{theorem}

The proofs of Theorem~\ref{Dirichlet} and Theorem~\ref{Poisson}
are built upon the fundamental works of W. Ballmann, F. Ledrappier and other mathematicians.
We extend their conclusions beyond non-positively curved manifolds.
Based on P. Eberlein's foundational work on the geometry of manifolds without conjugate points,
combined with our prior investigations of the dynamics of geodesic flows on such manifolds, we have overcome a series of geometric obstacles,
thereby extending the aforementioned mathematicians' seminal works to the framework of manifolds without conjugate points.

This paper is organized as follows.
In Section~\ref{geometry} we investigate the geometric properties of the manifolds without conjugate points.
In Section~\ref{BMS} we prove the Theorem~\ref{Dirichlet}.
In Section~\ref{question}, we prove the Theorem~\ref{Poisson}.

\section{\bf{Geometry of the Manifolds without Conjugate Points}}\label{geometry}

In this section, we will summarize all the notions that we needed in this paper, and we also investigate some geometric properties of the manifolds without conjugate points.

Let $M$ be a compact Riemannian manifold, and $\widetilde{M}$ be its universal covering manifold.
Let $\Gamma\triangleq\pi_1(M)$, then $M=\widetilde{M}/\Gamma$.
In this paper, we always assume that $M=\widetilde{M}/\Gamma$ is compact, and $\widetilde{M}$ is a simply connected uniform visibility manifold without conjugate points which satisfies the Axiom $2$, or $\widetilde{M}$ is a rank $1$ simply connected manifold without focal points.

As mentioned in Section~\ref{intro}, the concept of rank was initially introduced for non-positively curved manifolds.
Since the flat strip theorem is also valid for the manifolds without focal points, this concept can be generalized to the manifolds without focal points.

\begin{Defi}[cf. \cite{BBE}] \label{def1}
Let $M$ be a manifold without focal points, for each $v \in T^{1}M$, the \emph{rank} of $v$, denoted by \emph{\text{rank}($v$)}, is defined as the dimension of the vector space of parallel Jacobi fields along the geodesic $c_{v}$, and the rank of $M$ is defined as \emph{\text{rank}($M$):=$\min\{$\text{rank}$(v) \mid v \in T^{1}M\}$}. For a geodesic $c$, the rank of the geodesic $c$ is defined as \emph{\text{rank}($c$)}=\emph{\text{rank}($c'(t)$)}, $\forall\ t\in \mathbb{R}$.
The rank $1$ vectors are also called regular vectors, and the rank $1$ geodesics are called regular geodesics.
\end{Defi}

Throughout this paper we always assume that the geodesics are unit speed.
Two geodesics $c_{1}$ and $c_{2}$ in $\widetilde{M}$ are called $\mathbf{positively~asymptotic}$,
if there is a positive constant $C$ such that
$$d(c_1(t),c_2(t))\leq C, \quad\forall t\geq 0.$$
The positively asymptotic is an equivalence relation among geodesics on $\widetilde{M}$,
the set of the equivalence classes is called the $\mathbf{ideal~boundary}$ and is denoted by $\bf{\widetilde{M}(\infty)}$.

For the pinched negatively curved manifolds, it is well known that for each pair $(\xi,\eta)$ of $\widetilde{M}(\infty)$ with $\xi\neq\eta$,
there is one and only one geodesic connecting them. But when we consider non-positively curved manifolds, or more generally,
the manifolds without focal points or the manifolds without conjugate points, the situation becomes much more complicated:
for $\xi,\eta\in\widetilde{M}(\infty)$ with $\xi\neq\eta$, there may not exist a geodesic connecting them;
and even when the connecting geodesic do exist, it is not necessarily unique.
For the manifolds without focal points, the flat strip theorem implies that if there exist two connecting geodesics,
then there exists a family of infinitely many connecting geodesics, such a family of connecting geodesics forms a flat strip.

\begin{Defi} [cf. \cite{EO}] \label{def1}
Let $\widetilde{M}$ be a complete simply connected Riemannian manifold without conjugate points,
we call it satisfies \emph{Axiom~2}, if for any points $\xi,\eta \in \widetilde{M}(\infty)$ with $\xi\neq\eta$, there exists at most one geodesic connecting them.
\end{Defi}

In a sense, the condition of ``without conjugate points" is so broad that we need to add some additional conditions to derive important geometric properties
and fine (globally and locally) estimations. Visibility is such a condition, which was first introduced by Patrick Eberlein.

\begin{Defi} [cf. \cite{EO,Eb1}] \label{def2}
$\widetilde{M}$ is called a \emph{visibility~manifold} if for any $p\in\widetilde{M}$ and $\epsilon>0$, there exists a
constant $R_{p,\epsilon}>0$, such that for any geodesic (segment) $c:[a,b] \to \widetilde{M}$, $d(p,c)\geqslant R_{p,\epsilon}$ implies that $\measuredangle_p(c(a),c(b))\geqslant\epsilon$. Here we allow $a$ and $b$ to be infinity.
If the constant $R$ is independent of the choice of the point $p$,
$\widetilde{M}$ is called a \emph{uniform~visibility~manifold}. $M$ is called a
\emph{(uniform)~visibility~manifold} if its universal cover $\widetilde{M}$ does.
\end{Defi}

In contrast to Axiom $2$, the axiom of visibility insures that the existence of the connecting geodesic between two distinct points at infinity.


Visibility is a very natural condition. It is proved in \cite{BO} that the pinched negatively curved manifolds are uniform visibility.
On the other hand, there are many visibility manifolds admit some regions with zero and positive sectional curvature. For example,
a closed surface without conjugate points and genus greater than $1$ is a uniform visibility manifold.
For more information, see \cite{Eb1}, which contains almost all the useful geometric properties about visibility manifolds without conjugate points.

If $\widetilde{M}$ is an $n$-dimensional manifold without focal points, or a visibility manifold without conjugate points,
it is well known that (cf. \cite{OS,Eb1}) for any point $p\in\widetilde{M}$ and $\xi\in\widetilde{M}(\infty)$,
there exists a unique $v\in T^{1}_{p}\widetilde{M}$ such that
$c_{v}(+\infty)=\xi$, where $c_{v}$ is the unique geodesic satisfies $c(0)=p$ and $c'(0)=v$.
Therefore $\widetilde{M}(\infty)$ is homeomorphic to $T^{1}_{p}\widetilde{M}$,
which is homeomorphic to the $(n-1)$-dimensional unit sphere $\mathbb{S}^{n-1}$.

\begin{notation}\label{no1}
Given two different points $x, y\in\mathbf{\overline{\widetilde{M}}}\triangleq \widetilde{M}\cup\widetilde{M}(\infty)$
and let $\mathbf{c_{x,y}}$ be the unique geodesic connecting $x$ and $y$, if in addition $x\in \widetilde{M}$, we parametrize the geodesic $c_{x,y}$
by $c_{x,y}(0)=x$. We list the following notations of the angles between geodesics:
\begin{displaymath}
	\begin{aligned}
		\measuredangle_p(x,\bm{v}) & =    \measuredangle(c'_{p,x}(0),\bm{v}),\quad \bm{v}\in T_p^1\widetilde{M}, ~x\neq p; \\
        C(\bm{v},\epsilon)    & =         \Big\{x\in \overline{\widetilde{M}} - \{p\} ~\Big | ~\measuredangle_p(x,\bm{v})<\epsilon\Big\}, ~~\bm{v}\in T_p^1\widetilde{M};  \\                                                               C_\epsilon(\bm{v})    & =         \Big\{c_{\bm{w}}(+\infty)~\Big|~ \bm{w}\in T^1_p\widetilde{M},\measuredangle(\bm{v},\bm{w})<\epsilon\Big\}\subset\widetilde{M}(\infty), ~~\bm{v}\in T_p^1\widetilde{M};\\
		TC(\bm{v},\epsilon,r) & =    \Big\{x\in\overline{\widetilde{M}}~\Big| ~\measuredangle_p(x,\bm{v})<\epsilon, d(p,x)> r\Big\},~~\bm{v}\in T_p^1\widetilde{M}.
	\end{aligned}
\end{displaymath}
\end{notation}

The last one $TC(\bm{v},\epsilon,r)$ is called the $\mathbf{truncated~cone~with~axis~\bm{v}~and~angle~\epsilon}$.
For any point $\xi\in\widetilde{M}(\infty)$, the set of truncated cones containing this point actually forms local bases and hence, forms the bases for a topology $\tau$. This topology is unique and usually called the $\mathbf{cone~topology}$.
The cone topology is independent of the choice of the point p, according to the uniform visibility. Under the cone topology,
$\overline{\widetilde{M}}$ is homeomorphic to the n-dimensional unit closed ball in $\mathbb{R}^{n}$.
More precisely, for any $p\in \widetilde{M}$, let $B_{p}$ be the closed unit ball in the tangent space $T_{p}\widetilde{M}$, i.e.,
$B_{p} \triangleq \{ v \in T_{p}\widetilde{M} \mid \parallel v\parallel \leqslant 1\}$,
then if $M$ is a visibility manifold without conjugate points, the following map
\[
h:B_{p} \rightarrow \overline{\widetilde{M}},~~~ v \mapsto h(v) =
\left\{
\begin{aligned}
\exp_{p}(\frac{v}{1-\parallel v\parallel}) & \quad \text{if }  \parallel v \parallel < 1, \\
c_{v}(+\infty)~~~~~~~~~ & \quad \text{if } \parallel v \parallel = 1,
\end{aligned}
\right.
\]
is a homeomorphism under the cone topology.
For more details, refer to~\cite{Eb1,EO}.

Denote by $L(\Gamma)\triangleq\widetilde{M}(\infty)\cap\overline{\Gamma(p)}$,
where $\overline{\Gamma(p)}$ is the closure of the orbit of the $\Gamma$-ation at $p$ under the cone topology.
$L(\Gamma)$ is called the $\mathbf{limit~set}$ of $\Gamma$.
Due to the visibility axiom, one can check that $L(\Gamma)$ is independent of the choice of the point $p$.
$\Gamma$ is called $\mathbf{non}$-$\mathbf{elementary}$ if $\#L(\Gamma)=\infty$. In this paper, we always assume that $\Gamma$ is non-elementary.

Two (not necessarily distinct) points $\xi, \eta \in \widetilde{M}(\infty)$ are called $\mathbf{\Gamma}$-$\mathbf{dual}$,
if there exists a sequence $\{\alpha_{n}\}^{\infty}_{n=1}\subset \Gamma$ and a point $p \in \widetilde{M}$
(hence for all points $p \in \widetilde{M}$ due to visibility) such that under the cone topology,
$$\alpha^{-1}_{n}(p)\rightarrow \xi, ~~~\alpha_{n}(p)\rightarrow \eta.$$

We call an isometry $\alpha \in \mathrm{Iso}(\widetilde{M})$ an $\mathbf{axial}$ element
if there is a constant $T>0$ and a geodesic $c$ in $\widetilde{M}$ such that
$$ \alpha \circ c(t)=c(t+T),~~~t\in \mathbb{R}.$$
The geodesic $c$ is called an $\mathbf{axis}$ of the axial element $\alpha$.
An axial element $\alpha$ is called $\mathbf{rank~1}$ if the corresponding axis is rank $1$.

The following properties have been proved by Eberelin, O'Sullivan and us, which are useful later.

\begin{proposition}\label{pro20}
	Let $\widetilde{M}$ be a simply connected uniform visibility manifold without conjugate points,
or a rank $1$ simply connected manifold without focal points.
	\begin{enumerate}
   \item  The following map is continuous:
		\begin{displaymath}
			\begin{aligned}
				\Psi:~& T^1\widetilde{M}\times[-\infty,\infty] & \to     &~\widetilde{M}\cup\widetilde{M}(\infty), \\
				       & \qquad(\bm{v},t)                       & \mapsto &~c_{\bm{v}}(t).
			\end{aligned}
		\end{displaymath}
    \item  For any point $p \in \widetilde{M}$ and $v,w \in T_{p}^{1}\widetilde{M}$, if $v\neq w$,
                    then $$\lim_{t\rightarrow +\infty}d(c_{v}(t),c_{w}(t))=+\infty.$$
    \item  Any two points in $L(\Gamma)$ are $\Gamma$-dual.
    \item  Let $\alpha,~\beta$ be two elements in $\Gamma$.
Suppose that $\alpha$ is an axial element, and geodesic $c$ is an axis of $\alpha$, if $c(+\infty)$ is a fixed point of $\beta$,
then there exists a non-zero integer $n$ such that $\beta \circ \alpha^{n} = \alpha^{n} \circ \beta$,
and $c(-\infty)$ is another fixed point of $\beta$.
    \item If $\Gamma$ is non-elementary and cocompact, i.e., $\widetilde{M}/\Gamma$ is compact, then the $\Gamma$-action on the ideal boundary is minimal,
    i.e., for any $\xi \in \widetilde{M}(\infty)$, $\overline{\Gamma(\xi)}=\widetilde{M}(\infty)$.
    \item For any $\bm{v}\in T^1{\widetilde{M}}$ and positive constants $R,\epsilon$,
        there is a constant $L=L(\epsilon, R)$, such that for any $t>L$,
        \begin{displaymath}
            B(c_{\bm{v}}(t),R)\subset C(\bm{v},\epsilon).
        \end{displaymath}
        Here $B(c_{\bm{v}}(t),R)$ is the open ball centered at $c_{\bm{v}}(t)$ with radius $R$.
	\end{enumerate}
\end{proposition}
\begin{proof}
According to the statement of this proposition, we divide the discussions into two cases.
\begin{itemize}
		\item Case I: $\widetilde{M}$ is a uniform visibility manifold without conjugate points. \\
         In this case, the items $1$, $2$ and $3$ are proved by Eberlein in \cite{Eb1};  the items $4$ and $5$ are proved by us in \cite{L};
         the item $6$ follows directly from the definition of uniform visibility.
		\item Case II: $\widetilde{M}$ is a rank $1$ manifold without focal points. \\
        In this case, the item $2$ is proved by O'Sullivan in \cite{OS}; the items $1$, $3$, $5$ and $6$ are proved by us in \cite{LWW};
        and the item $4$ can be proved by the same approach as we used in our previous paper \cite{L} for the case of manifolds without conjugate points.
	\end{itemize}
\end{proof}

The following result plays a key role in the research of the dynamics of the manifolds beyond non-positive curvature.

\begin{lemma}\label{lem6}
Let $\widetilde{M}$ be a simply connected manifold without focal points,
or be a simply connected uniform visibility manifold without conjugate points and satisfies Axiom $2$,
let $v\in T^{1}\widetilde{M}$ be a unit tangent vector (when $\widetilde{M}$ has no focal points, we require that $rank(v)=1$),
then for any $\epsilon > 0$, there are neighborhoods $U_{\epsilon}$ of $c_{v}(-\infty)$ and $V_{\epsilon}$ of $c_{v}(+\infty)$ such that for each pair $(\xi,\eta)\in U_{\epsilon}\times V_{\epsilon}$, there exists a unique
(rank $1$) connecting geodesic $c=c_{\xi,\eta}$ with $d(c_{v}(0),c_{\xi,\eta})<\epsilon$.
\end{lemma}
\begin{proof}
According to the statement of this proposition, we divide the discussions into two cases.
\begin{itemize}
		\item Case I: $\widetilde{M}$ is a rank $1$ manifold without focal points. \\
        It is just the Proposition $7$ in our previous paper \cite{LWW}.
		\item Case II:  $\widetilde{M}$ is a uniform visibility manifold without conjugate points and satisfies Axiom $2$.\\
        In this case, the uniform visibility and Axiom $2$ insure that there is a unique connecting geodesic,
        so we only need to prove the inequality
        \begin{equation}\label{eq71.1}
		d(c_{v}(0),c_{\xi,\eta})<\epsilon.
       \end{equation}

        Suppose \eqref{eq71.1} is not true, then there is a positive real number $\epsilon_{0}$ and two sequences
        $\{\xi_{n}\}_{n\in \mathbb{N}^{+}}, ~\{\eta_{n}\}_{n\in \mathbb{N}^{+}}\subset \widetilde{M}(\infty)$ such that for all $n\in\mathbb{N}^+$,
        we have
        \begin{equation}\label{eq72.1}
		\xi_{n}\in TC(-v,\frac{1}{n},n), ~~\eta_{n}\in TC(v,\frac{1}{n},n),
       \end{equation}
       i.e.,
        $$
		\lim_{n\rightarrow +\infty}\xi_{n}=c_{v}(-\infty), ~~\lim_{n\rightarrow +\infty} \eta_{n}=c_{v}(+\infty),
        $$
and the connecting geodesic $c_{n}\triangleq c_{\xi_{n},\eta_{n}}$ satisfies
\begin{equation}\label{eq74}
		d(c_{v}(0),c_{n}) \geq \epsilon_{0}.
       \end{equation}

Take an arbitrary positive increasing number sequence $\{t_{n}\}_{n\in \mathbb{N}^{+}}$ such that $t_{n} > n$ for each $n\in \mathbb{N}^{+}$,
thus if we denote by $p_{n}\triangleq c_{v}(-t_{n})$, $q_{n}\triangleq c_{v}(t_{n})$, then $p_{n}\in TC(-v,\frac{1}{n},n),~q_{n}\in TC(v,\frac{1}{n},n)$.
By \eqref{eq72.1} we know that the geodesic $c_{n}$ will intersect the truncated cones $TC(-v,\frac{1}{n},n)$ and $TC(v,\frac{1}{n},n)$, respectively.

For every $n\in \mathbb{N}^+$, choose $p'_{n}\in c_{n}\cap TC(-v,\frac{1}{n},n)$ and $q'_{n}\in c_{n}\cap TC(v,\frac{1}{n},n)$.
Let $a_{n}:[0,1]\rightarrow \widetilde{M}$ be a smooth curve in $TC(-v,\frac{1}{n},n)$,
and connecting $p_{n}$ and $p'_{n}$ with $a_{n}(0)=p_{n}$ and $a_{n}(1)=p'_{n}$,
and $\theta_{s}\triangleq \measuredangle_{p}\Big(c_{v}(-\infty),a_{n}(s)\Big)$ is an increasing function of $s$, where $p=c_{v}(0)$.

\begin{figure}[hb]
		\centering
		\includegraphics[width=0.5\textwidth]{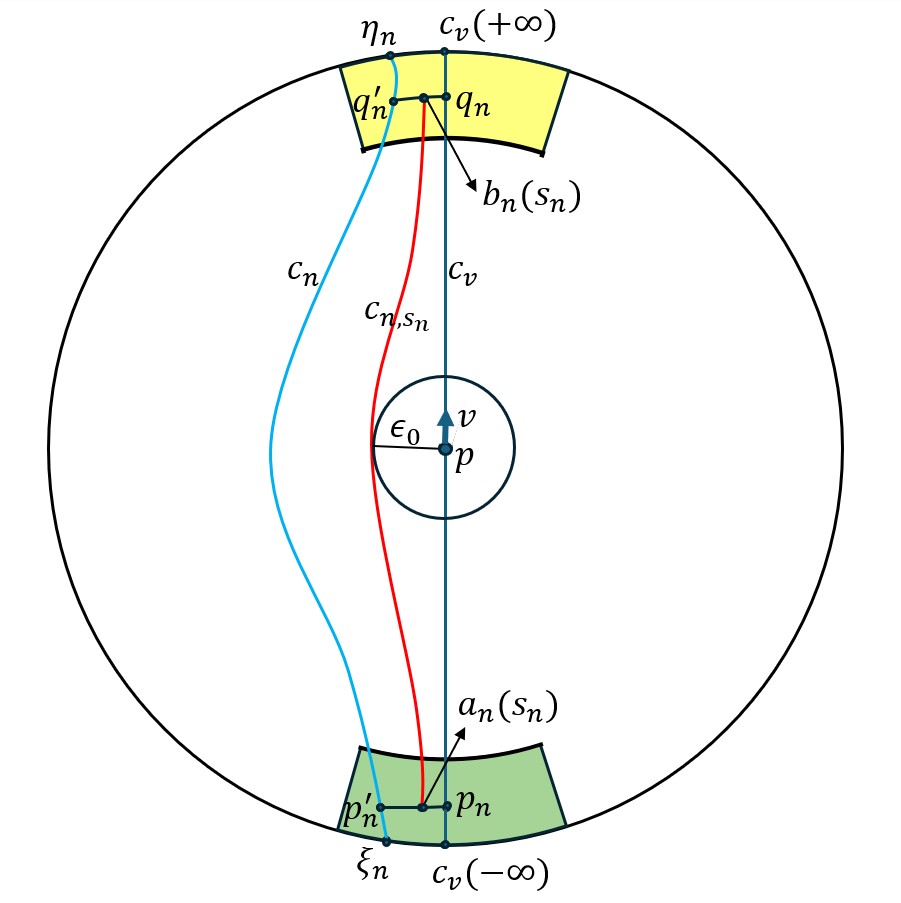}
\end{figure}

Similarly, let $b_{n}:[0,1]\rightarrow \widetilde{M}$ be a smooth curve in $TC(v,\frac{1}{n},n)$, and connecting $q_{n}$ and $q'_{n}$ with $b_{n}(0)=q_{n}$ and $b_{n}(1)=q'_{n}$, and $\vartheta_{s}\triangleq\measuredangle_{p}\Big(c_{v}(+\infty),b_{n}(s)\Big)$ is an increasing function of $s$.

Let $c_{n,s}$ ($s\in [0,1]$) be the connecting geodesic from $a_{n}(s)$ to $b_{n}(s)$,
with the parametrization $c_{n,0}(0)=c_{v}(0)$ and $c_{n,s}(0)$ is a continuous curve of $s$.

By the parametrization of $c_{n,0}$, $d(p,c_{n,0}(0))=d(p,p)=0$; and by \eqref{eq74}, $d(p,c_{n,1}(0))=d(p,c_{n}(0))\geq \epsilon_{0}$.
Therefore there exists a number $s_{n}\in (0,1]$ such that $d(p,c_{n,s_{n}}(0))=\epsilon_{0} >0$.

Passing to a subsequence if necessary we can assume that $\lim_{n\rightarrow +\infty}c'_{n,s_{n}}(0)= v_{\infty}\in T^{1}\widetilde{M}$.
Obviously $d(p,c_{v_{\infty}}(0))=\epsilon_{0}>0$, and $c_{v_{\infty}}(0) \notin c_{v}$.

Since $a_{n}(s_{n}) \in TC(-v,\frac{1}{n},n)$ and $b_{n}(s_{n}) \in TC(v,\frac{1}{n},n)$, by Proposition~\ref{pro20}(1), we know that
$$
c_{v_{\infty}}(-\infty)=\lim_{n\rightarrow+\infty}c_{n,s_{n}}(-\infty)=\lim_{n\rightarrow+\infty}a_{n}(s_{n})=c_{v}(-\infty),
$$
$$
c_{v_{\infty}}(+\infty)=\lim_{n\rightarrow+\infty}c_{n,s_{n}}(+\infty)=\lim_{n\rightarrow+\infty}b_{n}(s_{n})=c_{v}(+\infty),
$$
thus $c_{v_{\infty}}$ is a connecting geodesic from $c_{v}(-\infty)$ to $c_{v}(+\infty)$.

Because $c_{v_{\infty}}(0) \notin c_{v}$, the geodesics $c_{v}$ and $c_{v_{\infty}}$ are distinct,
which contradicts to the assumption that $\widetilde{M}$ satisfies Axiom $2$.
Thus  \eqref{eq71.1} is valid, and we complete the proof of this case.
\end{itemize}
\end{proof}

For any $\alpha_{1}, \alpha_{2} \in \Gamma$,  $\alpha_{1}$ and $\alpha_{2}$ are called $\mathbf{equivalent}$
if there exist $k_{1}, k_{2}\in \mathbb{Z} -\{0\}$ and $\beta \in \Gamma$ such that
$$\alpha^{k_{1}}_{1}=\beta \alpha^{k_{2}}_{2} \beta^{-1}.$$

If $\alpha_{1}$ and $\alpha_{2}$ are equivalent and $\alpha_{1}$ is an axial element with axis $c$,
then $\alpha_{2}$ is also an axial element with axis $\beta^{-1}(c)$. In fact
\begin{displaymath}
		\begin{aligned}
			\alpha^{k_{2}}_{2}(\beta^{-1}c(t)) & = \beta^{-1} \alpha^{k_{1}}_{1}\beta(\beta^{-1}c(t))    \\
			                                   & = \beta^{-1} \alpha^{k_{1}}_{1}(c(t))                    \\
			                                   & = \beta^{-1}c(t+k_{1}T).                                  \\
		\end{aligned}
\end{displaymath}
Therefore $\alpha_{2}(\beta^{-1}c(t))=\beta^{-1}c(t+\frac{k_{1}}{k_{2}}T)$.
Project through the canonical covering map from $\widetilde{M}$ to $M$,
the geodesics $c$ and $\beta^{-1}(c)$ in $\widetilde{M}$ corresponding to a same closed geodesic in $M$.
Thus the number of equivalence classes of $\Gamma$ provides a lower bound for the number of geometrically distinct closed geodesics in $M$.

\begin{theorem}\label{Ba}
(1) Let $\widetilde{M}$ be a complete Riemannian manifold without focal points, if there exists a rank $1$ axial isometry $\alpha$,
then either $\Gamma = \langle\alpha\rangle$ or $\Gamma$ has infinitely many equivalence classes of rank $1$ axial isometries.

(2) Let $\widetilde{M}$ be a simply connected uniform visibility manifold without conjugate points and satisfies Axiom $2$,
if there exists an axial isometry $\alpha$,
then either $\Gamma = \langle\alpha\rangle$ or $\Gamma$ has infinitely many equivalence classes of axial isometries.
\end{theorem}
\begin{proof}
(1) Let $c$ be an axis of the axial isometry $\alpha$, since $\alpha$ is rank $1$, the axes of $\alpha$ is unique.

If $\Gamma \neq \langle\alpha\rangle$, then $\{c(-\infty),c(+\infty)\}\varsubsetneqq L(\Gamma)$,
thus there exists $\xi \in L(\Gamma)-\{c(-\infty),c(+\infty)\}$. By the property of rank $1$ manifolds (cf. \cite{LWW}, Corollary 2),
there exists a unique rank $1$ geodesic from $c(-\infty)$ to $\xi$, and we denote it by $c_{\xi}$. For each $n\in \mathbb{N}^{+}$, let
\begin{displaymath}
		\begin{aligned}
			U_{n} & = TC \Big(-c^{'}_{\xi}(0),\frac{1}{n},n\Big)    \\
			      & = \Big\{p\in \overline{\widetilde{M}}-c_{\xi}(0) ~\Big| ~
                      \measuredangle_{c_{\xi}(0)}(p,c(-\infty))< \frac{1}{n}, ~d(p,c_{\xi}(0))>n\Big\},                    \\
            V_{n} & = TC \Big(c^{'}_{\xi}(0),\frac{1}{n},n\Big)    \\
			      & = \Big\{p\in \overline{\widetilde{M}}-c_{\xi}(0) ~\Big| ~
                      \measuredangle_{c_{\xi}(0)}(p,\xi)< \frac{1}{n},~d(p,c_{\xi}(0))>n\Big\}.                    \\
		\end{aligned}
\end{displaymath}

\begin{figure}[hb]
		\centering
		\includegraphics[width=0.5\textwidth]{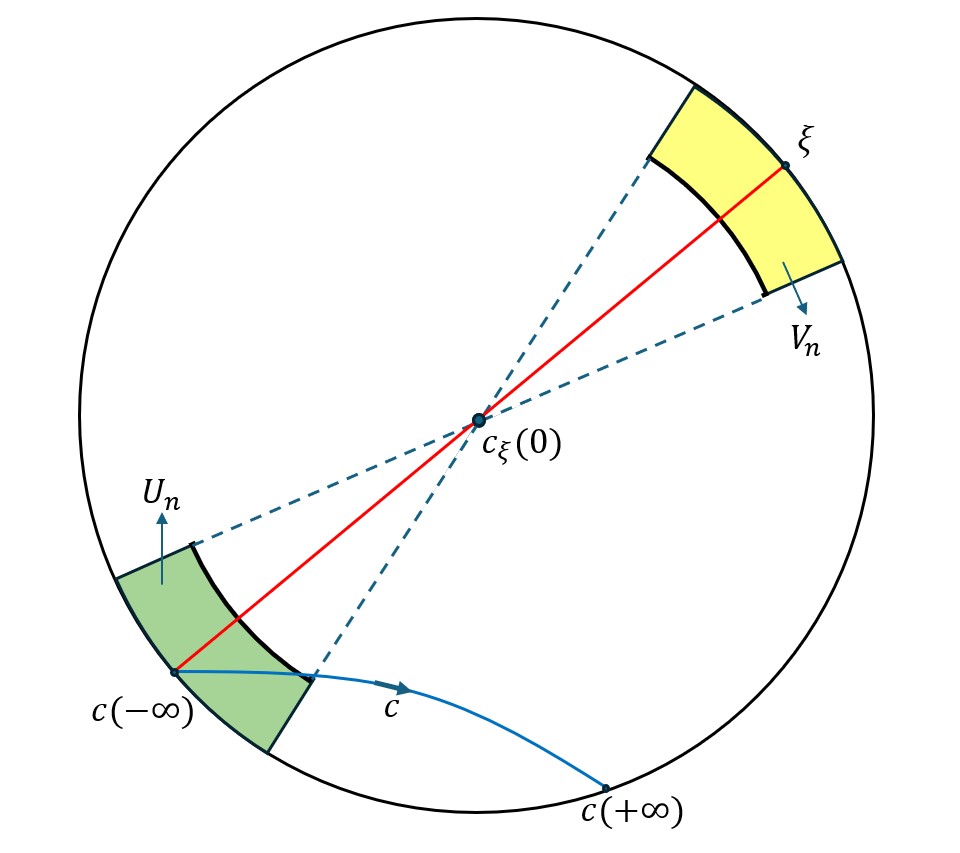}
\end{figure}

By Proposition~\ref{pro20}(3) we know that $c(-\infty)$ and $\xi$ are $\Gamma-$dual, thus for each $n\in \mathbb{N}^{+}$,
there exists $\alpha_{n}\in \Gamma$ such that
$$
\alpha_{n}(\overline{\widetilde{M}}-U_{n}) \subset V_{n},~~~\alpha^{-1}_{n}(\overline{\widetilde{M}}-V_{n}) \subset U_{n}.
$$
Since $\overline{V}_{n} \subset \overline{\widetilde{M}}-U_{n},~\overline{U}_{n} \subset \overline{\widetilde{M}}-V_{n}$, we have that
$$
\alpha_{n}(\overline{V}_{n}\cap\widetilde{M}(\infty)) \subset \overline{V}_{n}\cap\widetilde{M}(\infty),~~~
\alpha^{-1}_{n}(\overline{U}_{n}\cap\widetilde{M}(\infty)) \subset \overline{U}_{n}\cap\widetilde{M}(\infty).
$$

Since both of $\overline{U}_{n}\cap\widetilde{M}(\infty)$ and $\overline{V}_{n}\cap\widetilde{M}(\infty)$ are homeomorphic to closed $(\dim(M)-1)$-dimensional disks,
by the Brouwer fixed point theorem, $\alpha_{n}$ has a fixed point in $\overline{U}_{n}\cap\widetilde{M}(\infty)$ and $\overline{V}_{n}\cap\widetilde{M}(\infty)$, respectively.
And we denote them by $\xi_{n}$ and $\eta_{n}$, respectively. By Lemma~\ref{lem6}, for arbitrarily small $\epsilon >0$,
there exists an positive integer $N=N(\epsilon)>\frac{1}{\epsilon}$, for each $n\geqslant N$, there is a rank $1$ (thus unique) connecting geodesic $c_{n}\triangleq c_{\xi_{n},\eta_{n}}$,
and furthermore (we can shrink the set $U_{n}$ and $V_{n}$ if necessary),
\begin{equation}\label{eq65}
d(c_{n}(0),c_{\xi}(0))<\frac{1}{n}<\epsilon.
\end{equation}
Obviously the rank $1$ geodesic $c_{n}$ is the axis of the isometry $\alpha_{n}$.

By the definition, for $n$ large enough, $c(+\infty)\notin V_{n}$. Proposition~\ref{pro20} (4) implies that $\xi_{n}\neq c(-\infty)$,
on the other hand, $\lim_{n\rightarrow+\infty}\xi_{n}\rightarrow c(-\infty)$ indicates that we can assume that the elements of the sequence $\{\xi_{n}\}_{n\in \mathbb{N}^{+}}$
are pairwise distinct. Thus the geodesics $\{c_{n}\}_{n\in \mathbb{N}^{+}}$ are pairwise distinct.

If there are only finitely equivalence classes of rank $1$ axial isometries, passing to a subsequence if necessary, we can assume that for each $n\in \mathbb{N}^{+}$,
there exists $\beta_{n}\in \Gamma$ with $\beta_{n}(c_{n})=c_{1}$. Since the geodesics $\{c_{n}\}_{n\in \mathbb{N}^{+}}$ are pairwise distinct,
we know that $\{\beta_{n}\}_{n\in \mathbb{N}^{+}}$ are also pairwise distinct. Parametrized the geodesic $c_{n}$ by $\beta_{n}(c_{n}(0))=c_{1}(0)$.
\eqref{eq65} implies that
$$
d(\beta^{-1}_{n}(c_{1}(0)),c_{\xi}(0))<\frac{1}{n},~~~n\in \mathbb{N}^{+},
$$
which contradicts to the discreteness of the $\Gamma$. Thus $\Gamma$ has infinitely many equivalence classes of rank $1$ axial isometries.

(2) For the visibility manifolds without conjugate points, since all the required geometric properties have been established,
the method used in (1) can be directly applied to complete the proof.
\end{proof}

The following result says that the isometries act on $\overline{\widetilde{M}}$ as homeomorphisms for visibility manifolds without conjugate points,
which was proved for the non-positively curved visibility manifolds by P. Eberlein and B. O'Neill in ~\cite{EO}.

\begin{lemma}\label{lem161}
Let $\widetilde{M}$ be a simply connected uniform visibility manifold without conjugate points,
for any $\xi \in \widetilde{M}(\infty)$ and a sequence $\{p_{n}\}_{n\in \mathbb{N}^{+}}\subset \widetilde{M}$,
if $\lim_{n \rightarrow +\infty}p_{n}=\xi$ in the cone topology, then for each $\alpha \in \Gamma$,
we have $\lim_{n \rightarrow +\infty}\alpha(p_{n})=\alpha(\xi)$.
\end{lemma}
\begin{proof}
Fix a point $p\in \widetilde{M}$. First, we will show that
\begin{equation}\label{eq60}
		d(p, c_{p_{n},\xi})\rightarrow +\infty,~~~n\rightarrow\infty,
\end{equation}
where $ c_{p_{n},\xi}$ is the connecting geodesic ray from $p_{n}$ to $\xi$.

If~\eqref{eq60} is not true, by taking a sub-sequence if necessary, we can assume that there is a constant $R>0$, such that for each $n\in \mathbb{N}$,
there exists a $t_{n}>0$ satisfying
\begin{equation}\label{eq61}
		d(p, c_{p_{n},\xi}(t_{n})) \leqslant R, ~~~n\rightarrow\infty.
\end{equation}
Since $\lim_{n \rightarrow +\infty}p_{n}=\xi$, by the definition of the cone topology, we know that
\begin{equation}\label{eq62}
		t_{n}\rightarrow +\infty, ~~~n\rightarrow\infty.
\end{equation}

For each $n\in \mathbb{N}$, let $q_{n}\triangleq c_{p_{n},\xi}(t_{n})$. By \eqref{eq61},
we can assume that $q_{n}\rightarrow q \in \overline{B(p;R)}$ (by choosing a sub-sequence if needed).
Denote by $v_{n}\triangleq c^{'}_{p_{n},\xi}(t_{n})$ $(n \in \mathbb{N})$. As we mentioned before, all the geodesics in this paper are assumed to be unit speed,
thus the $\lim_{n\rightarrow +\infty}v_{n}$ exists without loss of generality, we denote this limit by $v$, i.e., $\lim_{n\rightarrow +\infty}v_{n}\triangleq v$.

By Proposition~\ref{pro20}(1) and \eqref{eq62},
$$
c_{v}(+\infty)=\lim_{n\rightarrow +\infty}c_{v_{n}}(+\infty)=\xi,
$$
$$
c_{-v}(+\infty)=\lim_{n\rightarrow +\infty}c_{-v}(t_{n})= \lim_{n\rightarrow +\infty}p_{n}=\xi.
$$
This contradicts to the divergence property  (Proposition~\ref{pro20}(2)).

Thus \eqref{eq60} is valid.

Since $\Gamma$ is a discrete subgroup of isometry group, for each $\alpha \in \Gamma$,
\begin{displaymath}
		\begin{aligned}
			  d(\alpha(p), c_{\alpha(p_{n}),\alpha(\xi)})& = d(\alpha(p), \alpha(c_{p_{n},\xi}))                              \\
			                                             & = d(p, c_{p_{n},\xi}) \rightarrow +\infty,~~~n\rightarrow\infty.    \\
		\end{aligned}
\end{displaymath}

Thus by triangle inequality we have
\begin{equation}\label{eq63}
d(p, c_{\alpha(p_{n}),\alpha(\xi)}) \geqslant d(\alpha(p), c_{\alpha(p_{n}),\alpha(\xi})) - d(p,\alpha(p))\rightarrow +\infty,~~~n\rightarrow\infty.
\end{equation} 	
Furthermore, \eqref{eq63} and the uniform visibility implies	
\begin{equation}\label{eq64}
\measuredangle_p(\alpha(p_{n}),\alpha(\xi))\rightarrow 0, ~~n\rightarrow\infty.
\end{equation}

By the definition of the cone topology, \eqref{eq63} and \eqref{eq64} show that $\lim_{n \rightarrow +\infty}\alpha(p_{n})=\alpha(\xi)$.
\end{proof}

\begin{proposition}\label{pr75}
(1) Let $\widetilde{M}$ be a simply connected manifold without focal points,
and $c$ is a rank $1$ geodesic axis of $\alpha \in \text{Iso}(X)$.
For all neighborhood $U \subset \overline{\widetilde{M}}$ of $c(-\infty)$ and $V \subset \overline{\widetilde{M}}$ of $c(+\infty)$,
there is $N=N(U,V) \in \mathbb{N}^+$ such that
$$\alpha^{n}(\overline{\widetilde{M}}-U)\subset V, ~~\alpha^{-n}(\overline{\widetilde{M}}-V)\subset U$$
for all $n \geq N$.

(2) Let $\widetilde{M}$ be a simply connected uniform visibility manifold without conjugate points, and $c$ is a geodesic axis of $\alpha \in \text{Iso}(X)$.
For all neighborhood $U \subset \overline{\widetilde{M}}$ of $c(-\infty)$ and $V \subset \overline{\widetilde{M}}$ of $c(+\infty)$,
there is $N=N(U,V) \in \mathbb{N}^+$ such that
$$\alpha^{n}(\overline{\widetilde{M}}-U)\subset V, ~~\alpha^{-n}(\overline{\widetilde{M}}-V)\subset U$$
for all $n \geq N$.
\end{proposition}

\begin{proof}
(1) This is proved by Watkins in \cite{Wa}.

(2) By Proposition~\ref{pro20}(6), we know that the conclusion holds for the point in $\widetilde{M}$.
Now it suffices to prove that the conclusion also holds for points in $\widetilde{M}(\infty)$,
this follows directly from Proposition~\ref{pro20}(1), Proposition~\ref{pro20}(6) and Lemma~\ref{lem161}.
\end{proof}

\begin{Defi}  \label{def10}
Let $\mu$ be a probability measure on $\Gamma$ such that the support of $\mu$ generates $\Gamma$ as a semi-group,
and $\nu$ be a probability measure on the ideal boundary $\widetilde{M}(\infty)$.
We define the \emph{convolution} of $\mu$ and $\nu$, denoted by $\mu\ast\nu$, as a probability measure on the ideal boundary by
$$\int_{\widetilde{M}(\infty)}f(\xi)d(\mu\ast\nu)(\xi)=\sum_{\alpha\in\Gamma}\Big(\int_{\widetilde{M}(\infty)}f(\alpha\xi)d\nu(\xi)\Big)\mu(\alpha),$$
where $f:\widetilde{M}(\infty) \rightarrow \mathbb{R}$ is a bounded measurable function on the ideal boundary.
\end{Defi}

In this paper, we are particularly interested in those probability measures on the ideal boundary $\widetilde{M}(\infty)$
that remain invariant under convolution with the probability measure $\mu$ on $\Gamma$.

We note that such measures always exist. In fact, by the weak compactness of the space of probability measures on the ideal boundary,
$\Big\{\frac{1}{n}\sum^{n-1}_{k=0}(\mu\ast)^{k}\nu\Big\}_{n\in \mathbb{N}^{+}}$ must contain a weakly convergent subsequence
whose weak limit is invariant under convolution. Thus from now on, we always assume that the measure $\nu$ satisfies
\begin{equation}\label{eq67}
\mu\ast\nu=\nu.
\end{equation}

\begin{lemma}\label{lem7}
Let $\mu$ and $\nu$ be probability measures as stated above and satisfies the equality~\eqref{eq67},
then for each $\xi\in \widetilde{M}(\infty)$, $\nu(\{\xi\})=0$.
\end{lemma}
\begin{proof}
Suppose that there exists an $\eta\in \widetilde{M}(\infty)$ with $\nu(\{\eta\})>0$, since $\nu$ is a probability measure, there exists
a $\xi_{0}\in \widetilde{M}(\infty)$ such that
$$
\nu(\{\xi_{0}\})=\sup_{\xi\in \widetilde{M}(\infty)}\nu(\{\xi\})>0.
$$
Thus for each $\alpha \in \Gamma$,
\begin{equation}\label{eq68}
\nu(\{\alpha^{-1}\xi_{0}\}) \leq \nu(\{\xi_{0}\}).
\end{equation}

By \eqref{eq67} and Definition~\ref{def10}, we have
\begin{displaymath}
		\begin{aligned}
			\nu(\{\xi_{0}\}) & = (\mu\ast\nu)(\{\xi_{0}\})                   \\
			                 & = \int_{\widetilde{M}(\infty)}\mathbf{1}_{\{\xi_{0}\}}(\eta)d(\mu\ast\nu)(\eta) \\
			                 & =  \sum_{\alpha\in\Gamma}\int_{\widetilde{M}(\infty)}\mathbf{1}_{\{\xi_{0}\}} (\alpha\eta)d\nu(\eta)d\mu(\alpha)    \\
                             & =   \sum_{\alpha \in \Gamma}\nu(\{\alpha^{-1}\xi_{0}\})\mu(\alpha).           \\
		\end{aligned}
\end{displaymath}
The above equality, and \eqref{eq68}, and the fact that ``$\mu$ is a probability measure on $\Gamma$" imply that
$$\nu(\{\alpha^{-1}\xi_{0}\}) = \nu(\{\xi_{0}\}),~~~~~~\alpha\in\Gamma.$$
This indicates that $\nu(\widetilde{M}(\infty))=+\infty$ since $\mathrm{card}(\Gamma)=+\infty$,
which contradicts to the fact that $\nu$ is a probability measure.
\end{proof}

We can get the following consequence directly from Lemma~\ref{lem7}.

\begin{corollary}\label{cor7}
Let $\mu$ and $\nu$ be probability measures as stated above and satisfies the equality~\eqref{eq67},
then for any $\epsilon>0$ and for each $\xi\in \widetilde{M}(\infty)$,
there exists a neighbourhood $U\subset\widetilde{M}(\infty)$ of $\xi$ such that $\nu(U)<\epsilon$.
\end{corollary}

\begin{proposition}\label{key}
Let $\widetilde{M}$ be a simply connected rank $1$ manifold without focal points,
or a simply connected uniform visibility manifold without conjugate points satisfying Axiom $2$,
and $\Gamma$ be a cocompact discrete subgroup of $Iso(\widetilde{M})$, i.e., $M \triangleq \widetilde{M}/\Gamma$ is compact,
assume that $\nu$ is a probability measure satisfies the equality~\eqref{eq67},
choose a point $p\in \widetilde{M}$, for a given $\epsilon > 0$, there exists a finite set $\Gamma_{0}$ in $\Gamma$ with the following property:

If $\{\alpha_{n}\}_{n\in \mathbb{N}^{+}} \subset \Gamma$ with $\alpha_{n}(p)\rightarrow \eta_{0}\in \widetilde{M}(\infty)$,
then there exists an open subset $U\subset \widetilde{M}(\infty)$ such that
\begin{itemize}
		\item $U \cap C\big(c'_{p,\eta_{0}}(0),\epsilon\big) = \emptyset$, where the cone $C\big(c'_{p,\eta_{0}}(0),\epsilon\big)$ is defined in the Notation~\ref{no1};
		\item $\nu(U)< \epsilon$.
	\end{itemize}
And there exists an element $\beta \in \Gamma_{0}$ and a subsequence $\{\alpha_{n_{k}}\}_{k\in \mathbb{N}^{+}} \subset \{\alpha_{n}\}_{n\in \mathbb{N}^{+}}$,
$$
\alpha_{n_{k}}\beta\big(\widetilde{M}(\infty)-U\big)\subset C\big(c'_{p,\eta_{0}}(0),\epsilon\big).
$$
\end{proposition}
\begin{proof}
(1) By Theorem~\ref{Ba}, we know that there exist two distinct rank $1$ axes $c_{1}$ and $c_{2}$,
the corresponding axial elements are $\beta_{1}$ and $\beta_{2}$, respectively.
For $i=1,2$, take $U^{\pm}_{i}\subset \overline{\widetilde{M}}$ the pairwise disjoint small neighbourhoods of $c_{i}(\pm\infty)$.
By Corollary~\ref{cor7} we can choose these neighbourhoods so small that for a given $\epsilon >0$,
$$
\nu\Big(U^{\pm}_{i}\cap \widetilde{M}(\infty)\Big)<\epsilon,~~~i=1,~2.
$$

By Lemma 7.3 of our previous paper~\cite{LLW}, we can choose $V^{\pm}_{i} \subset U^{\pm}_{i}$ such that
\begin{equation}\label{eq69}
\measuredangle_q\Big(\overline{\widetilde{M}}-U^{\pm}_{i}\Big) < \frac{\epsilon}{2},~~~q\in V^{\pm}_{i}, ~i=1,~2.
\end{equation}

Choose $\epsilon$ and $U^{\pm}_{i}$ small enough such that for any pair $(\xi,\eta)\in \widetilde{M}(\infty)\times \widetilde{M}(\infty)$,
at least one of the following four conditions holds:
$$
(\mathbf{I})~~~\xi \in \widetilde{M}(\infty) - U^{+}_{i},~~C\big(c'_{p,\eta}(0),\epsilon\big)\subset \overline{\widetilde{M}}-U^{-}_{i},~~~~~i=1,~2;
$$
$$
(\mathbf{II})~~~\xi \in \widetilde{M}(\infty) - U^{-}_{i},~~C\big(c'_{p,\eta}(0),\epsilon\big)\subset \overline{\widetilde{M}}-U^{+}_{i},~~~~~i=1,~2.
$$

\begin{figure}[hb]
		\centering
		\includegraphics[width=0.5\textwidth]{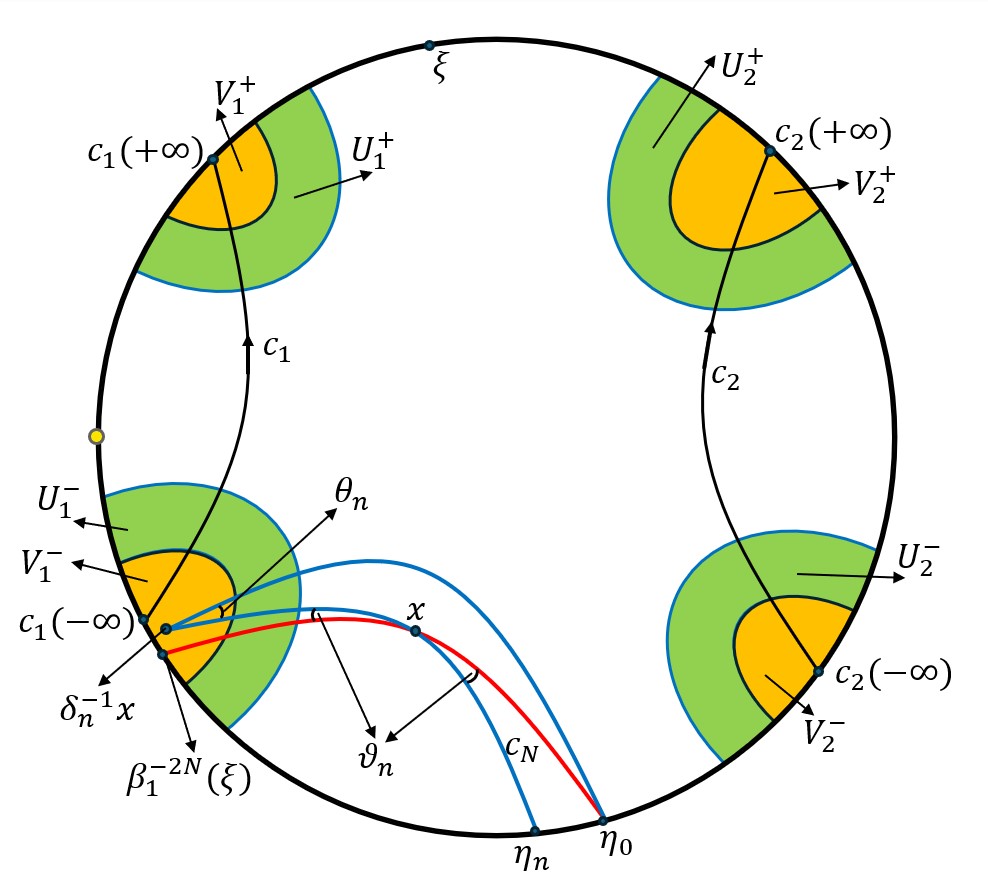}
\end{figure}

By Proposition~\ref{pr75}(1), there exists a number $N\in \mathbb{N}^{+}$, such that for any $n\geqslant N$, we have
\begin{equation}\label{eq72}
\beta_{i}^{n}(\overline{\widetilde{M}}-V^{-}_{i})\subset V^{+}_{i}, ~~\beta_{i}^{-n}(\overline{\widetilde{M}}-V^{+}_{i})\subset V^{-}_{i},~~~~i=1,~2.
\end{equation}

Let $\Gamma_{0}\triangleq \Big\{\beta_{1}^{N}, \beta_{2}^{N}, \beta_{1}^{-N}, \beta_{2}^{-N}\Big\}$.
By passing to a sub-sequence if necessary, we may assume that
$\lim_{n\rightarrow+\infty}\alpha^{-1}_{n}(p)\triangleq\xi \in \widetilde{M}(\infty)$ in the cone topology.
Without loss of generality, we can assume the condition $(\mathbf{I})$ holds for the pair $(\xi,\eta_{0})\in \widetilde{M}(\infty)\times \widetilde{M}(\infty)$,
with $i=1$, i.e.,
$$
\xi \in \widetilde{M}(\infty) - U^{+}_{1},~~C\big(c'_{p,\eta_{0}}(0),\epsilon\big)\subset \overline{\widetilde{M}}-U^{-}_{1}.
$$
In this case, denote by $\delta_{n}\triangleq \alpha_{n}\beta^{N}_{1}$,
\begin{equation}\label{eq70}
\measuredangle_p\Big(\delta_{n}\xi, \delta_{n}\big(\overline{\widetilde{M}}-U^{-}_{1}\big)\Big)
= \measuredangle_{\delta^{-1}_{n}p}\Big(\xi, \overline{\widetilde{M}}-U^{-}_{1}\Big)
< \frac{\epsilon}{2}.
\end{equation}
The last inequality uses \eqref{eq69} and the fact that
$\delta^{-1}_{n}p=\beta^{-N}_{1}\alpha^{-1}_{n}(p)\rightarrow\beta^{-N}_{1}(\xi)\in V^{-}_{1}$.


By \eqref{eq72}, $\beta^{-N}_{1}\xi\in V^{-}_{1}$ and $\beta^{N}_{1}\eta_{0}\in V^{+}_{1}$.
Choose $V^{+}_{1}$ and $V^{-}_{1}$ small enough such that the Lemma~\ref{lem6} holds,
then there is a unique connecting geodesic $c_{\beta^{-N}_{1}\xi, \beta^{N}_{1}\eta_{0}}$.
Let $c_{N}\triangleq \beta^{-N}_{1} c_{\beta^{-N}_{1}(\xi),\beta^{N}_{1}\eta_{0}}=c_{\beta^{-2N}_{1}(\xi),\eta_{0}}$ and $x\triangleq c_{N}(0)$,
then
\begin{equation}\label{eq71}
\measuredangle_{x}\big(\delta_{n}\eta_{0}, \eta_{0}\big)
                                            \leq  \measuredangle_{x}\big(\delta_{n}\eta_{0}, \delta_{n}x\big) +
                                                        \measuredangle_{x}\big(\delta_{n}x, \eta_{0}\big).
\end{equation}

Now we want to show that $\measuredangle_{x}\big(\delta_{n}\eta_{0}, \eta_{0}\big)\rightarrow 0$ from the inequality \eqref{eq71}.
Because $x$ is a point independent of $n$, we can further conclude that $\lim_{n\rightarrow +\infty}\delta_{n}\eta_{0}=\eta_{0}$ in the cone topology.

On the one hand, by Proposition~\ref{pro20}(6), we know that $\delta_{n}x=\alpha_{n}\big(\beta^{N}_{1}x\big)\rightarrow \eta_{0}$.
Therefore, the second term on the right-hand side of inequality \eqref{eq71} satisfies $\measuredangle_{x}\big(\delta_{n}x, \eta_{0}\big)\rightarrow 0$.

On the other hand, Proposition~\ref{pro20}(6) implies that $\alpha^{-1}_{n}x\rightarrow\xi \in \widetilde{M}(\infty) - U^{+}_{1}$,
thus $\delta^{-1}_{n}x=\beta^{-N}_{1}\alpha^{-1}_{n}x \in V^{-}_{1}$. For each $n \in \mathbb{N}^{+}$,
denote by
\begin{displaymath}
		\begin{aligned}
        \eta_{n}    & \triangleq  c_{\delta^{-1}_{n}x,x}(+\infty), ~~~\theta_{n}\triangleq  \measuredangle_{\delta^{-1}_{n}x}\big(x, \eta_{0}\big), \\
     \vartheta_{n}  & \triangleq  \measuredangle_{x}\big(\beta^{-N}_{1}\xi, \delta^{-1}_{n}x\big)
                                   = \measuredangle_{x}\big(\eta_{0}, \eta_{n}\big).\\
\end{aligned}
\end{displaymath}
By Lemma~\ref{lem161}, $\delta^{-1}_{n}x=\beta^{-N}_{1}\alpha^{-1}_{n}x \rightarrow \beta^{-N}_{1}\xi$ in the cone topology,
thus $\vartheta_{n}\rightarrow 0$. By Lemma 2.9 in \cite{Wa}, $\theta_{n}\leqslant \vartheta_{n}\rightarrow 0$.

We conclude that $\measuredangle_{x}\big(\delta_{n}\eta_{0}, \eta_{0}\big)\rightarrow 0$,
thus $\lim_{n\rightarrow +\infty}\delta_{n}\eta_{0}=\eta_{0}$.
Finally, by \eqref{eq70} it follows that this Proposition holds when taking $U\triangleq \widetilde{M}(\infty) \cap U^{-}_{1}$.

(2) For the visibility manifolds without conjugate points,
all geometric properties hold except for Lemma 2.9 in \cite{Wa} that used in the proof of (1).
Thus we only need to prove $\measuredangle_{x}\big(\delta_{n}\eta_{0}, \eta_{0}\big)\rightarrow 0$.
In fact, since $\eta_{0}\in \overline{\widetilde{M}}-U^{-}_{1}$, we can choose the parametrization of $c_{N}$
such that $x \triangleq c_{N}(0) \in \overline{\widetilde{M}}-U^{-}_{1}$, then by the definition of uniform visibility and
$\delta^{-1}_{n}x=\beta^{-N}_{1}\alpha^{-1}_{n}x \rightarrow \beta^{-N}_{1}\xi\subset V^{-}_{1}$, we can conclude that
$\measuredangle_{x}\big(\delta_{n}\eta_{0}, \delta_{n}(x)\big)=\measuredangle_{\delta^{-1}_{n}x}\big(\eta_{0}, x\big)\rightarrow 0$,
thus $\measuredangle_{x}\big(\delta_{n}\eta_{0}, \eta_{0}\big) \leq
\measuredangle_{x}\big(\delta_{n}\eta_{0}, \delta_{n}x\big) + \measuredangle_{x}\big(\delta_{n}x, \eta_{0}\big) \rightarrow 0$.
\end{proof}

\section{\bf The Dirichlet Problem at Infinity}\label{BMS}
\setcounter{equation}{0}\setcounter{theorem}{0}

Now we investigate the Dirichlet problem at infinity and the proof of Theorem \ref{Dirichlet} will be given in this section.

\subsection{The Construction of the Probability Measure $\mu$ on $\Gamma$}\label{definenu}

We now give a construction of the probability measure $\mu$ on $\Gamma$ due to Lyons and Sullivan.

\begin{lemma}\label{lemmaLSmeasure}\cite{LS}
Suppose $X$ is a discrete set in $\widetilde{M}$, for each $x\in X$, there exist a relatively compact set
$E_x$ and an open set $B_x$, such that

\begin{enumerate}
  \item $\overline{E}_x\subset B_x$ for all $x\in X$.
  \item $\bigcup\limits_{x\in X}E_x=\widetilde{M}$.
  \item The supremum of the Harnack constants of the pairs $(B_x,E_x)$ is bounded, that is, there exists a constant $C<\infty$, such that
  for any positive harmonic function $h_x$ on $B_x$,
  \begin{equation}
  \sup\limits_{x\in X}\left\{\sup\limits_{y,z\in E_x}\left\{\frac{h_x(y)}{h_x(z)}\right\}\right\}\leq C.
  \end{equation}

\end{enumerate}
Then there exists a family of probability measures $\{\tilde{\mu}_y\}_{y\in\widetilde{M}}$ on $X$ such that
\begin{enumerate}
  \item $\tilde{\mu}_y(x)>0, \forall x\in X.$
  \item For any positive harmonic function $h:\widetilde{M}\rightarrow \mathbb{R}$,
  \begin{equation}\label{lsmeasure(2)}
  h(y)=\sum\limits_{x\in X}h(x)\tilde{\mu}_y(x).
  \end{equation}
  \item If $\alpha: \widetilde{M}\rightarrow \widetilde{M}$ is an isometry and $\alpha(X)=X$, then
  \begin{equation}\label{lsmeasure(3)}
  \tilde{\mu}_{\alpha y}(\alpha x)=\tilde{\mu}_y(x),\forall x\in X,y\in\widetilde{M}.
  \end{equation}
\end{enumerate}
\end{lemma}

We choose a fixed point $p\in \widetilde{M}$ and let $X=\Gamma(p)$ in our situation. Choose a relatively compact Borel set $E$ such that
$$\bigcup\limits_{\alpha\in \Gamma}\alpha(E)=\widetilde{M},~\mbox{and}~\alpha(E)\cap\beta(E)=\emptyset,\forall\alpha\neq\beta \in \Gamma.$$
Choose an open geodesic ball $B$ about $p$ such that $\overline{E}\subset B$. The family of sets pairs $(\alpha(B),\alpha(E))_{\alpha\in\Gamma}$ satisfies
the conditions of Lemma \ref{lemmaLSmeasure} since $\Gamma$ acts by isometries.

Denote by $\mathcal{E}_q^U$ the exit distribution from $U$ for Brownian motion starting at $q\in U$, where $U\subset \widetilde{M}$ be any open, relatively compact
domain with smooth boundary $\partial U$. Then the support of $\mathcal{E}_q^U$ is in $\partial U$ and
$\mathcal{E}_q^U$ is a probability measure on $\widetilde{M}$.

Denote by $\tilde{\nu}^U$ the balayage of measure $\tilde{\nu}$ on $\widetilde{M}$ with support in $U$, that is
\begin{equation}
\tilde{\nu}^U(\cdot)=\int_{q\in\widetilde{M}}\mathcal{E}_q^U(\cdot)\tilde{\nu}({\rm d} q).
\end{equation}
Hence for any bounded measurable function $f$,
\begin{equation}
\tilde{\nu}^U(f)\triangleq\int_{x\in\widetilde{M}}f(x)\tilde{\nu}^U({\rm d}x)=\int_{x\in\widetilde{M}}\int_{q\in\widetilde{M}}
f(x)\mathcal{E}_q^U({\rm d}x)\tilde{\nu}({\rm d}q).
\end{equation}
By the mean value property of harmonic functions, if $f$ is harmonic,
\begin{equation}\label{varepsilonf=fq}
  \mathcal{E}_q^U(f)=\int_{x\in\widetilde{M}}f(x)\mathcal{E}_q^U({\rm d}x)= \int_{x\in\partial{U}}f(x)\mathcal{E}_q^U({\rm d}x)= f(q),
\end{equation}
since the support of $\mathcal{E}_q^U$ is in $\partial U$.
Then
\begin{eqnarray}\label{muuf=muf}
  \tilde{\nu}^U(f)&=&\int_{q\in\widetilde{M}}\left(\int_{x\in\widetilde{M}}f(x)\mathcal{E}_q^U({\rm d}x)\right)\tilde{\nu}({\rm d}q)\nonumber\\
 & =&\int_{q\in\widetilde{M}}\mathcal{E}_q^U(f)\tilde{\nu}({\rm d}q)\nonumber\\&=&\int_{q\in\widetilde{M}}f(q)\tilde{\nu}({\rm d}q)=\tilde{\nu}(f).
\end{eqnarray}

First we construct $\mu_y$ in the case $y=\alpha_0 p,$ where $\alpha_0\in\Gamma$.   Let $\nu_0=\mathcal{E}_{\alpha_0 p}^{\alpha_0(B)}$, it is clear that $\nu_0$ is a finite positive measure on $\widetilde{M}$. Construct inductively finite positive measures $\lambda_n,\nu_n,n=1,2,\cdots$ by
\begin{equation}\label{lambdan+1}
\lambda_{n+1}=\sum_{\alpha\in\Gamma}\frac{1}{C}\nu_n(\alpha(E))\mathcal{E}_{\alpha p}^{\alpha (B)},~~\nu_{n+1}=\sum_{\alpha\in\Gamma}
(\nu_n|_{\alpha(E)})^{\alpha(B)}-\lambda_{n+1},
\end{equation}
where $\nu_n|_{\alpha(E)}$ is the restriction of $\nu_n$ to $\alpha(E)$ and $C$ is the Harnack constant defined in Lemma \ref{lemmaLSmeasure}.
For any positive harmonic function $h$ on $\widetilde{M}$ which is $\nu_0$-summable, then $h$ is also $\lambda_{n}$-summable
and $\nu_n$-summable for $n=1,2,\cdots$, and by (\ref{varepsilonf=fq}), (\ref{muuf=muf}), (\ref{lambdan+1}) and the definition of the Harnack constant $C$,
\begin{eqnarray}\label{lambda+mun+1}
 \nu_{n+1}(h)+\lambda_{n+1}(h)&=&\int_{x\in\widetilde{M}}h(x)\sum_{\alpha\in\Gamma}
(\nu_n|_{\alpha(E)})^{\alpha(B)}({\rm d}x)\nonumber\\
&=&\sum_{\alpha\in\Gamma}\int_{x\in\widetilde{M}}h(x)
(\nu_n|_{\alpha(E)})^{\alpha(B)}({\rm d}x)\nonumber\\
&=&\sum_{\alpha\in\Gamma}\int_{x\in\widetilde{M}}h(x)
(\nu_n|_{\alpha(E)})({\rm d}x)\nonumber\\
&=&\sum_{\alpha\in\Gamma}\int_{x\in\alpha(E)}h(x)
\nu_n({\rm d}x)\nonumber\\
&=&\int_{x\in\widetilde{M}}h(x)
\nu_n({\rm d}x)=\nu_n(h),
\end{eqnarray}
\begin{eqnarray}\label{lambdan+1geqc2}
\lambda_{n+1}(h)&=&\int_{x\in\widetilde{M}}h(x)\sum_{\alpha\in\Gamma}\frac{1}{C}\nu_n(\alpha(E))\mathcal{E}_{\alpha p}^{\alpha (B)}({\rm d}x)\nonumber\\
&=&\sum_{\alpha\in\Gamma}\frac{1}{C}\nu_n(\alpha(E))\int_{x\in\widetilde{M}}h(x)\mathcal{E}_{\alpha p}^{\alpha (B)}({\rm d}x)\nonumber\\
&=&\sum_{\alpha\in\Gamma}\frac{1}{C}\nu_n(\alpha(E))h(\alpha p)=\frac{1}{C^2}\sum_{\alpha\in\Gamma}\int_{x\in\alpha(E)}Ch(\alpha p)\nu_n({\rm d}x)\nonumber\\
&\geq&\frac{1}{C^2}\sum_{\alpha\in\Gamma}\int_{x\in\alpha(E)}h(x)\nu_n({\rm d}x)\nonumber\\
&=&\frac{1}{C^2}\int_{x\in\widetilde{M}}h(x)\nu_n({\rm d}x)=\frac{1}{C^2}\nu_n(h).
\end{eqnarray}
Denote by $\mathcal{E}_y$ the family of measures on $\widetilde{M}$ which is the Dirac measure at $y$, that is $\mathcal{E}_y(y)=1,\mathcal{E}_y(z)=0,\forall y\neq z\in\widetilde{M}.$
Let $\tau_{n+1}^{}=\sum_{\alpha\in\Gamma}\frac{1}{C}\nu_n(\alpha(E))\mathcal{E}_{\alpha p}, n=0,1,2,\cdots$, then the measures $\tau_n^{}$ are supported on $\Gamma(p)$.
If $h$ is a $\nu_0$-summable positive harmonic function on $\widetilde{M}$, by (\ref{varepsilonf=fq}), (\ref{lambdan+1}), (\ref{lambda+mun+1}) and (\ref{lambdan+1geqc2}), we obtain
\begin{eqnarray}\label{halphap=munh+}
  &&h(\alpha_0 p)=\mathcal{E}_{\alpha_0 p}^{\alpha_0(B)}(h)=\nu_0(h)\nonumber\\
  &=&\nu_1(h)+\lambda_1(h)=\nu_2(h)+\lambda_2(h)+\lambda_1(h)=\nu_3(h)+\lambda_3(h)+\lambda_2(h)+\lambda_1(h)=\cdots\nonumber\\
  &=&\nu_n(h)+\sum_{k=1}^{n}\lambda_k(h)=\nu_n(h)+\sum_{k=1}^{n}\left(\sum_{\alpha\in\Gamma}\frac{1}{C}\nu_{k-1}(\alpha(E))\mathcal{E}_{\alpha p}^{\alpha (B)}(h)\right)\nonumber\\
  &=&\nu_n(h)+\sum_{k=1}^{n}\left(\sum_{\alpha\in\Gamma}\frac{1}{C}\nu_{k-1}(\alpha(E))h({\alpha p})\right)\nonumber\\
  &=&\nu_n(h)+\sum_{k=1}^{n}\left(\sum_{\alpha\in\Gamma}\frac{1}{C}\nu_{k-1}(\alpha(E))\mathcal{E}_{\alpha p}(h)\right)=\nu_n(h)+\sum_{k=1}^{n}\tau_k^{}(h).
\end{eqnarray}
Since by  (\ref{lambda+mun+1}) and (\ref{lambdan+1geqc2}),
$$\nu_{n-1}(h)=\nu_n(h)+\lambda_{n}(h)\geq\nu_n(h)+\frac{1}{C^2}\nu_{n-1}(h),$$
we have
\begin{equation}\label{}
  \nu_n(h)\leq\left(1-\frac{1}{C^2}\right)\nu_{n-1}(h)\leq\cdots\leq\left(1-\frac{1}{C^2}\right)^n\nu_{0}(h)=\left(1-\frac{1}{C^2}\right)^nh(\alpha_0 p).
\end{equation}
Let $n\rightarrow\infty$ in (\ref{halphap=munh+}), we obtain
\begin{equation}\label{nualphapdef}
  h(\alpha_0 p)=\sum_{k=1}^{\infty}\tau_k^{}(h)\triangleq\mu_{\alpha_0 p}(h).
\end{equation}
Let $h\equiv1$ in (\ref{nualphapdef}) and note that the measures $\tau_k$ are supported on $\Gamma(p)$, we obtain that
$\mu_{\alpha_0 p}=\sum_{k=1}^{\infty}\tau_k^{}$ is a probability measure on $X=\Gamma(p)$.

For $y\notin\Gamma(p)$, the construction of $\mu_y$ is similar to the case of $\mu_{\alpha_0 p}$, we just need to let $\nu_0=\mathcal{E}_y$,
the Dirac measure at $y$.
Similarly, we have
$$ h(y)=\sum_{k=1}^{\infty}\tau_k^{}(h)\triangleq\mu_y(h).$$

\begin{Defi}
For any $y\in \widetilde{M},$ define
\begin{equation}\label{}
  \tilde{\mu}_y=\sum_{k=1}^{\infty}\tau_k,
\end{equation}
where $\nu_0=\mathcal{E}_y,$ $\nu_n$ defined by (\ref{lambdan+1}) and $\tau_{n+1}^{}=\sum_{\alpha\in\Gamma}\frac{1}{C}\nu_n(\alpha(E))\mathcal{E}_{\alpha p}, n=0,1,2,\cdots.$
\end{Defi}
By the definition of $\tau_k$, we obtain that $\tilde{\mu}_y$ is a probability measure supported on $\Gamma(p).$

The family of probability measures $\tilde{\mu}_y$ on $\Gamma(p)$ defined above satisfy Lemma \ref{lemmaLSmeasure}. In fact,
for any positive harmonic function $h$ on $\widetilde{M}$ and any isometry $\beta:\widetilde{M}\rightarrow\widetilde{M}$ such that $\beta(\Gamma(p))=\Gamma(p),$ we have
$$h(y)=\sum_{k=1}^{\infty}\tau_k^{}(h)=\tilde{\mu}_y(h)=\sum_{\alpha\in \Gamma}h(\alpha p)\tilde{\mu}_y(\alpha p),$$
and
\begin{eqnarray*}
   \tilde{\mu}_y(\alpha p)&=&\sum_{k=1}^{\infty}\tau_k^{}(\alpha p)
   =\sum_{k=1}^{\infty}\sum_{\alpha'\in\Gamma}\frac{1}{C}\nu_{n-1}(\alpha'(E))\mathcal{E}_{\alpha'p}(\alpha p) \\
  &=& \sum_{k=1}^{\infty}\frac{1}{C}\nu_{n-1}(\alpha(E))
   = \sum_{k=1}^{\infty}\frac{1}{C}\nu_{n-1}(\beta\alpha(E))\\
   &=&\sum_{k=1}^{\infty}\sum_{\alpha'\in\Gamma}\frac{1}{C}\nu_{n-1}(\alpha'(E))\mathcal{E}_{\alpha'p}(\beta\alpha p) \\
&=&\sum_{k=1}^{\infty}\tau_k^{}(\beta\alpha p)
=\tilde{\mu}_{\beta y}(\beta \alpha p).
\end{eqnarray*}
Thus (\ref{lsmeasure(2)}) and (\ref{lsmeasure(3)}) hold.

Let $\mu_p(\alpha)=\tilde{\mu}_p(\alpha p),\forall \alpha \in \Gamma.$ Then $\mu_p$ is a probability measure on $\Gamma$ and depend on the choice of $p$. In the following, we will denote by $\mu$ the probability measure $\mu_p$ for a fixed point $p\in\widetilde{M}$.

\begin{Defi}\label{defmup}
Define the probability measure $\mu$ on $\Gamma$ by
\begin{equation}\label{equationmup}
\mu(\alpha)\triangleq\mu_p(\alpha)=\tilde{\mu}_p(\alpha p),\forall \alpha \in \Gamma.
\end{equation}
\end{Defi}

By the definition of $\tilde{\mu}_y$, the probability measure $\mu$ satisfies the condition in Definition \ref{def10}, that is,  the support of $\mu$ generates $\Gamma$ as a semi-group.
Define the $i$-th moment of $\mu$ by $\sum_{\alpha \in \Gamma}|\alpha|^i\mu(\alpha),$ where $|\alpha|=d(p,\alpha p)$.

\begin{lemma}\label{finitemoment}
  The measure $\mu$ defined above has finite moments of all orders. In fact, $\sum_{\alpha\in\Gamma}{\rm e}^{a|\alpha|}\mu(\alpha)<+\infty$
  for $0<a<\frac{-1}{2R_B}\ln\left(1-\frac1C\right)$, where $R_B$ is the radius of the geodesic ball $B$.
\end{lemma}

\begin{proof}
By the construction of $\mu_p$, we have $\nu_0=\mathcal{E}_p^B$.
  Let $\nu_{0,e}=\mathcal{E}_p^B$ and $\nu_{0,\alpha}=0$ for $\alpha\neq e$, where $e$ is the identity element of $\Gamma$.
  Then we have $\nu_0=\sum_{\alpha\in\Gamma}\nu_{0,\alpha}$. Define $\lambda_{n+1,\alpha}$ and $\nu_{n+1,\alpha},n=0,1,2,\cdots$ inductively by
  \begin{equation}\label{}
\lambda_{n+1,\alpha}=\sum_{\alpha'\in\Gamma}\frac{1}{C}\nu_{n,\alpha'}(\alpha(E))\mathcal{E}_{\alpha p}^{\alpha (B)},~~\nu_{n+1,\alpha}=\sum_{\alpha'\in\Gamma}
(\nu_{n,\alpha'}|_{\alpha(E)})^{\alpha(B)}-\lambda_{n+1,\alpha}.
\end{equation}
Then $\lambda_{n+1,\alpha}$ and $\nu_{n+1,\alpha}$ are supported on $\partial(\alpha (B))$ and inductively by (\ref{lambdan+1}) we have
\begin{equation}\label{}
  \lambda_{n+1}=\sum_{\alpha\in\Gamma}\lambda_{n+1,\alpha},~~\nu_{n+1}=\sum_{\alpha\in\Gamma}\nu_{n+1,\alpha},~~n=0,1,2,\cdots.
\end{equation}

  If $|\alpha'\alpha^{-1}|>2R_B$,
  $$d(\alpha\alpha^{-1}p,\alpha'\alpha^{-1}p)=d(p,\alpha'\alpha^{-1}p)=|\alpha'\alpha^{-1}|>2R_B,$$
  thus $\nu_{n,\alpha'}(\alpha(E))=0$ since $\nu_{n+1,\alpha'}$ are supported on $\partial(\alpha' (B))$.

  If  $|\alpha'\alpha^{-1}|\leq 2R_B$,
  \begin{eqnarray*}
    |\alpha| &=&d(p,\alpha p)=d(\alpha^{-1}p,p)=d(\alpha'\alpha^{-1}p,\alpha'p) \\
     &\leq& d(\alpha'\alpha^{-1}p,p)+d(p,\alpha'p)=|\alpha'\alpha^{-1}|+|\alpha'| \leq 2R_B+|\alpha'|.
      \end{eqnarray*}
    Since $\{\alpha(E)\}_{\alpha\in\Gamma}$ form a partition of $\widetilde{M}$, and
    $$(\nu_{n,\alpha'}|_{\alpha(E)})^{\alpha(B)}(\partial (\alpha(B)))=\int_{x\in\alpha(B)}\mathcal{E}_x^{\alpha(B)}(\partial (\alpha(B)))(\nu_{n,\alpha'}|_{\alpha(E)})({\rm d}x)=\nu_{n,\alpha'}(\alpha(E)),$$
     by the definition of $\lambda_{n+1,\alpha}$ and $\nu_{n+1,\alpha}$, we have
    \begin{eqnarray*}
      && b_{n+1}\\
      &\triangleq&\sum_{\alpha\in\Gamma}{\rm e}^{a|\alpha|}\nu_{n+1,\alpha}(\partial(\alpha (B)))  \\
       &=& \sum_{\alpha',\alpha\in\Gamma}{\rm e}^{a|\alpha|}\left(\nu_{n,\alpha'}(\alpha (E))-\frac1C\nu_{n,\alpha'}(\alpha (E))\right) \\
       &=& \left(1-\frac1C\right)\sum_{\alpha'\in\Gamma}\left(\sum_{\alpha\in\Gamma,|\alpha'\alpha^{-1}|\leq2R}{\rm e}^{a|\alpha|}\nu_{n,\alpha'}(\alpha (E))+
              \sum_{\alpha\in\Gamma,|\alpha'\alpha^{-1}|>2R}{\rm e}^{a|\alpha|}\nu_{n,\alpha'}(\alpha (E))\right)\\
       &\leq&{\rm e}^{2aR_B}\left(1-\frac1C\right)\sum_{\alpha'\in\Gamma}{\rm e}^{a|\alpha'|}\nu_{n,\alpha'}(\partial(\alpha' (B)))\\
       &=&{\rm e}^{2aR_B}\left(1-\frac1C\right)b_n.
          \end{eqnarray*}
 Since $a<\frac{-1}{2R_B}\ln\left(1-\frac1C\right)$ implies ${\rm e}^{2aR_B}\left(1-\frac1C\right)<1$, we have $\sum_{n=0}^{\infty}b_n<\infty$.
  By the definition of $\mu$ and $\mu_p$, we have
  $$\mu(\alpha)=\mu_p(\alpha p)=\sum_{n=0}^\infty\frac1C\nu_n(\alpha(E))=\sum_{n=0}^\infty\sum_{\alpha'\in\Gamma}
  \frac1C\nu_{n,\alpha'}(\alpha(E))=\sum_{n=0}^\infty\lambda_{n+1,\alpha}(\partial(\alpha(B))),$$
  and
  $$c_{n+1}\triangleq\sum_{\alpha\in\Gamma}{\rm e}^{a|\alpha|}\lambda_{n+1,\alpha}(\partial(\alpha (B))) \leq\frac1C{\rm e}^{2aR_B}b_n.$$
  Then
  $$\sum_{\alpha\in\Gamma}{\rm e}^{a|\alpha|}\mu(\alpha)=\sum_{n=0}^{\infty}\sum_{\alpha\in\Gamma}{\rm e}^{a|\alpha|}\lambda_{n+1,\alpha}(\partial(\alpha(B)))
  \leq\frac1C{\rm e}^{2aR_B}\sum_{n=0}^{\infty}b_n<\infty.$$
  \end{proof}

\subsection{The Convergence of $\mu$-harmonic Functions on $\Gamma$}

 Let $\mu$ be the probability measure on $\Gamma$ defined above.

 \begin{Defi}
 A function $h$ on $\Gamma$ is said to be $\mu$-harmonic, if
 \begin{equation}\label{}
  h(\alpha)=\sum_{\alpha'\in\Gamma}h(\alpha\alpha')\mu(\alpha').
 \end{equation}
 \end{Defi}
  That is to say, the value of $\mu$-harmonic function equal to its mean value with respect to $\mu$.
  Let $\nu$ be a probability measure on the ideal boundary $\widetilde{M}(\infty)$, such that $\mu\ast\nu=\nu$.
For any bounded measurable function $f$ on $\widetilde{M}(\infty)$, define $h_f$ on $\Gamma$ by
\begin{equation}\label{}
  h_f(\alpha)=\int_{\xi\in\widetilde{M}(\infty)}f(\alpha \xi)\nu({\rm d}\xi)=\int_{\xi\in\widetilde{M}(\infty)}f(\xi)(\alpha_*\nu)({\rm d}\xi),
\end{equation}
where $\alpha_*\nu$ is the measure on $\widetilde{M}(\infty)$ defined by $(\alpha_*\nu)(\cdot)=\nu(\alpha^{-1}\cdot)$.

\begin{lemma}
$h_f$ is a $\mu$-harmonic function for any bounded measurable function $f$ on $\widetilde{M}(\infty)$.
\end{lemma}

\begin{proof}
\begin{eqnarray*}
  \sum_{\alpha'\in\Gamma}h_f(\alpha\alpha')\mu(\alpha') &=&\sum_{\alpha'\in\Gamma} \int_{\xi\in\widetilde{M}(\infty)}f(\alpha\alpha' \xi)\nu({\rm d}\xi)\mu(\alpha')\\
  &=& \sum_{\alpha'\in\Gamma} \int_{\eta\in\widetilde{M}(\infty)}f(\alpha\eta)\nu({\rm d}(\alpha'^{-1}\eta))\mu(\alpha') \\
   &=&\int_{\eta\in\widetilde{M}(\infty)}f(\alpha\eta)(\mu\ast\nu)({\rm d}\eta)\\
    &=&\int_{\eta\in\widetilde{M}(\infty)}f(\alpha\eta)\nu({\rm d}\eta)
   =h_f(\alpha).
   \end{eqnarray*}
\end{proof}

\begin{lemma}\label{corollary1.7}
Let $\{\alpha_{n_{k}}\}_{k\in \mathbb{N}^{+}} \subset \Gamma$, $\eta_{0}\in \widetilde{M}(\infty)$, and $\beta\in \Gamma_0$ are as in Proposition \ref{key},
and $f$ be a positive bounded measurable function on $\widetilde{M}(\infty)$, then
$$\liminf_{k\rightarrow \infty}h_f(\alpha_{n_{k}}\beta)\geq (1-\epsilon)\inf\left(f|_{C(c'_{p,\eta_{0}}(0),\epsilon)}\right),$$
$$\limsup_{k\rightarrow \infty}h_f(\alpha_{n_{k}}\beta)\leq \sup\left(f|_{C(c'_{p,\eta_{0}}(0),\epsilon)}\right)+\sup(f)\epsilon.$$
\end{lemma}

\begin{proof}
By Proposition \ref{key},
\begin{eqnarray*}
   h_f(\alpha_{n_{k}}\beta)&=&\int_{\xi\in\widetilde{M}(\infty)}f(\alpha_{n_{k}}\beta \xi)\nu({\rm d}\xi)  \\
   &=&  \int_{\xi\in\widetilde{M}(\infty)-U}f(\alpha_{n_{k}}\beta \xi)\nu({\rm d}\xi)+\int_{\xi\in U}f(\alpha_{n_{k}}\beta \xi)\nu({\rm d}\xi)\\
   &\geq& (1-\epsilon)\inf\left(f|_{C(c'_{p,\eta_{0}}(0),\epsilon)}\right),
\end{eqnarray*}
and
\begin{eqnarray*}
   h_f(\alpha_{n_{k}}\beta)&=&\int_{\xi\in\widetilde{M}(\infty)}f(\alpha_{n_{k}}\beta \xi)\nu({\rm d}\xi)  \\
   &=&  \int_{\xi\in\widetilde{M}(\infty)-U}f(\alpha_{n_{k}}\beta \xi)\nu({\rm d}\xi)+\int_{\xi\in U}f(\alpha_{n_{k}}\beta \xi)\nu({\rm d}\xi)\\
   &\leq&\sup\left(f|_{C(c'_{p,\eta_{0}}(0),\epsilon)}\right)+\sup(f)\epsilon.
\end{eqnarray*}
\end{proof}

\begin{theorem}\label{hfgamman}
If $[\omega]=(\omega_0,\omega_1,\omega_2,\cdots)$ is a sequence in $\Gamma^{\mathbb{N}}$ such that $\lim_{n\rightarrow\infty}\omega_np=\eta_0\in\widetilde{M}(\infty)$
for $p\in\widetilde{M}$. Then $\lim_{n\rightarrow\infty}h_f(\omega_n)=f(\eta_0)$
for any bounded measurable function $f$ on $\widetilde{M}(\infty)$ which is continuous at $\eta_0$.
\end{theorem}

\begin{proof}
Define the probability measure $\mu^k$ on $\Gamma$ inductively by
$$\mu^k(\alpha)=\sum_{\alpha'\in\Gamma}\mu^{k-1}(\alpha')\mu(\alpha'^{-1}\alpha), k=1,2,\cdots, \mu_0=\mu.$$
If $h$ is $\mu$-harmonic, inductively we can obtain that $h$ is $\mu^k$-harmonic for all $k\in\mathbb{N}.$ In fact, if we assume that $h$
is $\mu^{k-1}$-harmonic, then
\begin{eqnarray*}
  h(\alpha_0) &=& \sum_{\alpha\in\Gamma}h(\alpha_0\alpha)\mu^{k-1}(\alpha)=\sum_{\alpha',\alpha\in\Gamma}h(\alpha_0\alpha\alpha')\mu(\alpha')\mu^{k-1}(\alpha) \\
   &=&\sum_{\alpha'',\alpha\in\Gamma}h(\alpha_0\alpha'')\mu(\alpha^{-1}\alpha'')\mu^{k-1}(\alpha)  \\
   &=& \sum_{\alpha''\in\Gamma}h(\alpha_0\alpha'')\sum_{\alpha\in\Gamma}\mu(\alpha^{-1}\alpha'')\mu^{k-1}(\alpha)  \\
   &=&\sum_{\alpha''\in\Gamma}h(\alpha_0\alpha'')\mu^{k}(\alpha''),\forall \alpha_0\in\Gamma.
\end{eqnarray*}
 For any $\alpha\in\Gamma$, there exists $k\in\mathbb{N}$ such that $\mu^k(\alpha)>0$ since the support of $\mu$ generates $\Gamma$.

Without loss of generality, we assume $f>0$. Suppose now that $(\omega_0,\omega_1,\omega_2,\cdots)$ is a sequence in $\Gamma$
such that $\lim_{n\rightarrow\infty}\omega_np=\eta_0$ but $\lim_{n\rightarrow\infty}h_f(\omega_n)=f(\eta_0)+\epsilon_0$ for some $\epsilon_0>0.$
For given $\eta_0$, let
$$\epsilon_1=\sup\{\epsilon_0>0~|~\exists[\omega]=(\omega_0,\omega_1,\cdots)\in\Gamma^{\mathbb{N}},\omega_np\rightarrow\eta_0, \mbox{but}~
h_f(\omega_n)\rightarrow  f(\eta_0)+\epsilon_0\}.$$
Choose $[\omega]=(\omega_0,\omega_1,\cdots)\in\Gamma^{\mathbb{N}}$ such that
$\lim_{n\rightarrow\infty}h_f(\omega_n)=f(\eta_0)+\epsilon_1$. Since $\eta_0$ is a continuous point of $f$, we can choose $\epsilon$ sufficiently small such that
\begin{equation}\label{equationiii}
  f(\eta_0)+\epsilon_1>(1+\epsilon)\big(f(\eta_0)+\frac{\epsilon_1}{2}\big)+\sup(f)\epsilon,
\end{equation}
and
\begin{equation}\label{equationiv}
  \sup\left(f|_{C(c'_{p,\eta_{0}}(0),\epsilon)}\right)\leq f(\eta_0)+\frac{\epsilon_1}{2}.
\end{equation}
By the choice of $\epsilon_1$, we have
\begin{equation}\label{equationvi}
  \limsup h_f(\omega_n\alpha')\leq f(\eta_0)+\epsilon_1,\forall \alpha'\in\Gamma,
\end{equation}
since for any $\alpha'\in\Gamma$, it holds $\omega_n\alpha'p\rightarrow \eta_0.$

Choose $\beta$ as in Lemma \ref{corollary1.7}. Then there exists $k\in\mathbb{N}$ such that $\mu^k(\beta)\triangleq a_\beta>0.$
Since $h_f$ is a $\mu$-harmonic function, it is also $\mu^k$-harmonic. Thus for all $n$,
\begin{eqnarray*}
h_f(\omega_n)&=&\sum_{\alpha'\in\Gamma}h_f(\omega_n\alpha')\mu^k(\alpha')\\
&=&\sum_{\alpha'\in\Gamma-\{\beta\}}h_f(\omega_n\alpha')\mu^k(\alpha')+h_f(\omega_n\beta)\mu^k(\beta).
\end{eqnarray*}
Let $n\rightarrow \infty$, by (\ref{equationiv}), (\ref{equationvi}) and Lemma \ref{corollary1.7}, we obtain
\begin{eqnarray}\label{}
  f(\eta_0)+\epsilon_1&\leq& (f(\eta_0)+\epsilon_1)(1-a_\beta) +a_\beta\left(\sup\left(f|_{C(c'_{p,\eta_{0}}(0),\epsilon)}\right)+\sup(f)\epsilon\right)\nonumber\\
  &\leq&(f(\eta_0)+\epsilon_1)(1-a_\beta) +a_\beta\left(f(\eta_0)+\frac{\epsilon_1}{2}+\sup(f)\epsilon\right).
 \end{eqnarray}
That is
\begin{equation}\label{}
   f(\eta_0)+\epsilon_1\leq f(\eta_0)+\frac{\epsilon_1}{2}+\sup(f)\epsilon.
\end{equation}
By (\ref{equationiii}), we have
$$
(1+\epsilon)\big(f(\eta_0)+\frac{\epsilon_1}{2}\big)+\sup(f)\epsilon <f(\eta_0)+\epsilon_1\leq
f(\eta_0)+\frac{\epsilon_1}{2}+\sup(f)\epsilon,$$
That is $1+\epsilon<1$, which contradicts $\epsilon>0.$
Hence we have $\limsup_{n\rightarrow\infty}h_f(\omega_n)\leq f(\eta_0)$ for any sequence which satisfies $\omega_np\rightarrow \eta_0$.
Similarly, we can also obtain $\liminf_{n\rightarrow\infty}h_f(\omega_n)\geq f(\eta_0)$ for any sequence which satisfies $\omega_np\rightarrow \eta_0$,
and the theorem follows.
\end{proof}

\subsection{The Convergence of Random Walk on $\Gamma$}\label{randomwalk}

Choose a fixed point $p\in \widetilde{M}$ and let $\mu$ be the probability measure on $\Gamma$ defined in Definition \ref{defmup}. We now define the left-invariant random walk on $\Gamma$.
Let sample space $\Omega=\Gamma^{\mathbb{N}}$ with elements $[\omega]=(\omega_0,\omega_1,\omega_2,\cdots),\omega_0=e.$
The probability measure $Q$ on $\Omega$ is defined by $\mu$ as follows: For given $\alpha_1,\alpha_2,\cdots,\alpha_k\in\Gamma,$
\begin{equation}\label{}
  Q\{[\omega]|\omega_1=\alpha_1,\omega_2=\alpha_2,\cdots,\omega_k=\alpha_k\}=\mu(\alpha_1)\mu(\alpha_2)\cdots\mu(\alpha_k)  .
\end{equation}

Define $\Gamma$-valued random variable sequence on $\Omega$ by
\begin{equation}\label{gammakomegaa}
  \gamma_k([\omega])=\gamma_k((\omega_0,\omega_1,\omega_2,\cdots))=\omega_{0}\omega_1\cdots\omega_{k-1},k=1,2,\cdots,
\end{equation}
then $[\gamma]=(\gamma_0,\gamma_1,\gamma_2\cdots)$ is a random walk on $\Gamma$ with transition probability from $\alpha\in\Gamma$ to $\beta\in\Gamma$
\begin{equation}\label{}
  Q\{\gamma_{k+1}=\beta|\gamma_k=\alpha\}=\mu(\alpha^{-1}\beta).
\end{equation}
The random walk on $\Gamma$ is transient since the support of $\mu$ generate $\Gamma$ (cf. \cite{Fu1}). Thus
\begin{equation}\label{gammakrightarrowinfty}
  |\gamma_k([\omega])|\triangleq d(p,\gamma_k([w])p)\rightarrow \infty, {\rm a.s.}~ [\omega] \in\Omega.
\end{equation}

\begin{lemma}\label{martingale}
For any bounded $\mu$-harmonic function $h$ on $\Gamma$, $M_k([\omega])=h(\gamma_k([\omega])),k=1,2,\cdots$ defines a martingale.
\end{lemma}

\begin{proof}
The integrability of $M_k$ follows since $h$ is bounded. And because $h$ is $\mu$-harmonic, we have
\begin{eqnarray*}
  E(M_{k+1}|M_k) &=&E(h(\gamma_{k+1})|\gamma_k) \\
   &=& \sum_{\alpha\in\Gamma} Q(\gamma_{k+1}=\alpha|\gamma_k)h(\alpha)\\
   &=& \sum_{\alpha\in\Gamma} h(\alpha)\mu(\gamma_{k}^{-1}\alpha)\\
   &=& \sum_{\alpha'\in\Gamma} h(\gamma_k\alpha')\mu(\alpha')=h(\gamma_k)=M_k.
\end{eqnarray*}
Thus  $\{M_k([\omega])\}$ is a martingale.
\end{proof}

\begin{theorem}\label{gammanp}
Choose a fixed point $p\in \widetilde{M}$. Let $\mu$ be the probability measure on $\Gamma$ defined in Definition \ref{defmup}
 and $\nu$ be the  probability measure on $\widetilde{M}(\infty)$ as in the Definition \ref{def10} which satisfies (\ref{eq67}),
 then for $\mu^{\otimes\mathbb{N}}$-a.e. $[\omega]\in\Omega$, there exists a unique $\xi([\omega])\in\widetilde{M}(\infty)$, such that
  $$\lim_{n\rightarrow\infty}\gamma_n([\omega])p=\xi([\omega]).$$
Furthermore, the hitting probability equals to $\nu$, therefore $\nu$ is uniquely determined.
\end{theorem}

\begin{proof}
By (\ref{gammakrightarrowinfty}), for $\mu^{\otimes\mathbb{N}}$-a.e. $[\omega]\in\Omega$,
there exists a subsequence of $\{\gamma_n([\omega])p\}$ converges to a point in $\widetilde{M}(\infty)$. We will prove that
all the subsequences of $\{\gamma_n([\omega])p\}$ converge to a same point in $\widetilde{M}(\infty)$.
Thus $\{\gamma_n([\omega])p\}$ is convergent.

Denote by $B(\xi,\varepsilon)=C\big(c'_{p,\xi}(0),\varepsilon\big)\cap\widetilde{M}(\infty),\xi\in\widetilde{M}(\infty).$
For a given $[\omega]\in\Omega,$ if there exist subsequences $\gamma_{k'}([\omega])p\rightarrow\xi\in\widetilde{M}(\infty)$ and
$\gamma_{k''}([\omega])p\rightarrow\eta\in\widetilde{M}$, $\xi\neq\eta.$
Choose sufficiently small $\varepsilon>0$ such that $B(\xi,\varepsilon)\cap B(\eta,\varepsilon)=\emptyset.$
There is a bounded continuous function $f$ on $\widetilde{M}(\infty)$ which is negative on $B(\xi,\varepsilon)$, positive on $B(\eta,\varepsilon),$
and $0$ otherwise.
 Then by Theorem \ref{hfgamman}, $h_f(\gamma_{k'}([\omega]))\rightarrow f(\xi)<0$ and $h_f(\gamma_{k''}([\omega]))\rightarrow f(\eta)>0.$
However, by martingale convergence theorem and Lemma \ref{martingale}, $h_f(\gamma_k([\omega]))$ converges almost surely. Thus
\begin{equation}\label{}
  Q\big\{[\omega]|\gamma_kp~\mbox{has clusterpoints in both} ~B(\xi,\varepsilon) \mbox{and} ~B(\eta,\varepsilon)\big\}=0.
\end{equation}

We can choose countably many $\xi_i\in\widetilde{M}(\infty)$ and $\varepsilon_j>0$ such that for any $\xi\neq\eta\in\widetilde{M}(\infty),$
there exist $\xi_{i_1},\xi_{i_2},\varepsilon_{j_0}$ which satisfy
$$\xi\in B(\xi_{i_1},\varepsilon_{j_0}),\eta\in B(\xi_{i_2},\varepsilon_{j_0}),\mbox{and}~B(\xi_{i_1},\varepsilon_{j_0})\cap B(\xi_{i_2},\varepsilon_{j_0})=\emptyset.$$
Hence, the probability of the set
$$\big\{[\omega]|\{\gamma_k([\omega]) p\}~\mbox{has at least two distinct cluster points in}~\widetilde{M}(\infty)\big\}$$
is
\begin{eqnarray*}
   &&Q \big\{[\omega]|\{\gamma_n([\omega]) p\}~\mbox{has at least two distinct cluster points in}~\widetilde{M}(\infty)\big\} \\
   &\leq& Q\left(\bigcup_{i_1,i_2,j}\big\{\{\gamma_k([\omega]) p\}~ \mbox{has cluster points in both} ~B(\xi_{i_1},\varepsilon_j) \mbox{and} ~B(\xi_{i_2},\varepsilon_j)\big\}
\right) \\
  &=& 0.
\end{eqnarray*}
Thus the first assertion follows.

Define $\pi_k([\omega])=\gamma_k([\omega])p$, then $\pi_k$ is measurable. Its pointwise limit
$$\pi([\omega])=\lim_{k\rightarrow\infty}\gamma_k([\omega])p$$
(in the sense of a.s.) is also measurable. For a given Borel set $\mathcal{A}\subset\widetilde{M}(\infty)$, the set
$$\{[\omega]|\lim_{k\rightarrow\infty}\gamma_k([\omega])p\in \mathcal{A}\}$$
is measurable. Then the hitting probability is given by
$$\tilde{\nu}(\mathcal{A})=Q\{[\omega]|\lim_{k\rightarrow\infty}\gamma_k([\omega])p\in \mathcal{A}\}.$$

For any bounded continuous function $f$ on $\widetilde{M}(\infty)$, we have $f(\lim_{k\rightarrow\infty}\gamma_k([\omega])p)
=\lim_{k\rightarrow\infty}h_f(\gamma_k([\omega]),{\rm a.s.}$ by Theorem \ref{hfgamman}. By the definition of $h_f$, the martingale convergence theorem
and Theorem \ref{hfgamman}, we obtain
\begin{eqnarray*}
  \int_{\xi\in\widetilde{M}(\infty)}f(\xi)\tilde{\nu}({\rm d}\xi) &=&\int_{[\omega]\in\Omega}\lim_{k\rightarrow\infty}h_f(\gamma_k([\omega])Q({\rm d}[\omega])   \\
 &=& \int_{[\omega]\in\Omega}h_f(\gamma_1([\omega])Q({\rm d}[\omega])  \\
  &=& h_f(e) \int_{[\omega]\in\Omega}Q({\rm d}[\omega])\\
  &=& \int_{\xi\in\widetilde{M}(\infty)}f(e\xi){\nu}({\rm d}\xi) =\int_{\xi\in\widetilde{M}(\infty)}f(\xi){\nu}({\rm d}\xi).
\end{eqnarray*}
The second assertion follows.
\end{proof}

By Theorem \ref{hfgamman} and \ref{gammanp}, we can obtain the following result.
\begin{corollary}
   For any continuous function $f$ on $\widetilde{M}(\infty)$, there exists a $\mu$-harmonic function $h_f$ on $\Gamma$,
  such that for $\mu^{\otimes \mathbb{N}}$-a.e. $[\omega]\in\Omega$, we have
  $$f(\lim_{k\rightarrow\infty}\gamma_k([\omega])p)
=\lim_{k\rightarrow\infty}h_f(\gamma_k([\omega]).$$
\end{corollary}

\subsection{The Dirichlet Problem at $\widetilde{M}(\infty)$}

Choose a fixed point $p\in \widetilde{M}$ and let $\mu$ be the probability measure on $\Gamma$ defined in Definition \ref{defmup},
that is,
 $\mu(\alpha)=\mu_p(\alpha)=\tilde{\mu}_p(\alpha p),\forall \alpha \in \Gamma.$

 If $h$ is a positive harmonic function on $\widetilde{M}$, let
 $\tilde{h}(\alpha)=h(\alpha p),\forall \alpha\in\Gamma.$ Then by the definition of $\mu_p$, we have
 \begin{eqnarray*}
 \tilde{h}(\alpha)&=&h(\alpha p)=\sum_{\alpha'p\in\Gamma(p)}h(\alpha'p)\mu_{\alpha p}(\alpha'p)\\
 &=&\sum_{\alpha''\in\Gamma}h(\alpha\alpha''p)\mu_{\alpha p}(\alpha\alpha''p)\\
 &=&\sum_{\alpha''\in\Gamma}\tilde{h}(\alpha\alpha'')\mu_{ p}(\alpha''p)\\
 &=&\sum_{\alpha''\in\Gamma}\tilde{h}(\alpha\alpha'')\mu(\alpha'').
 \end{eqnarray*}
Thus $\tilde{h}$ is a $\mu$-harmonic function on $\Gamma$.

If vice versa, for any $\mu$-harmonic function $\tilde{h}$ on $\Gamma$, there exists a harmonic function $h$ on $\widetilde{M}$ such that
$\tilde{h}(\alpha)=h(\alpha p),\forall \alpha\in\Gamma.$ Then for any continuous function $f$ on $\widetilde{M}(\infty)$, $h_f$
is a $\mu$-harmonic function on $\Gamma,$ there exists a harmonic function $h$ on $\widetilde{M}$, such that
$h_f(\alpha)=h(\alpha p),\forall \alpha\in\Gamma.$
By Theorem \ref{hfgamman} and \ref{gammanp}, if $\gamma_np\rightarrow \xi\in\widetilde{M}(\infty)$,
we have $h(\gamma_np)=h_f(\gamma_n)\rightarrow f(\xi).$
However, the existence of a correspondence between $\mu$-harmonic functions on $\Gamma$ and harmonic functions on $\widetilde{M}$ is quite unclear.
 In the following, we show that the Brownian path converges at $\widetilde{M}(\infty)$, and the hitting probability is given by the hitting probability of the
 random walk on $\Gamma$ defined in Subsection \ref{randomwalk}.

Denote by $\Omega$ the space of all sequences on $\Gamma$ and by $W$ the space of all $\widetilde{M}$-valued continuous functions, that is,
$$W=\{c|c:(0,\infty)\rightarrow \widetilde{M}~\mbox{is continuous.}\}$$
Denote by $P$ the probability measure on $W$ associated to Brownian motion starting at $p$ and by $Q$ the probability measure on $\Omega=\Gamma^{\mathbb{N}}$
associated to the $\mu$ random walk on $\Gamma$.
Then there is a probability measure $\overline{P}$ on $W\times\Omega$ such that the following properties hold (cf. \cite{LS}).
\begin{itemize}
  \item (i) The natural projection $W\times\Omega\rightarrow W$ maps $\overline{P}$ onto $P$.
  \item (ii) There exists a map $W\times\Omega\rightarrow\Gamma^{\mathbb{N}}:(c,w)\rightarrow[\gamma(c,w)]=(\gamma_0(c,w),\gamma_1(c,w),\gamma_2(c,w),\cdots)$
  which maps $\overline{P}$ onto $Q$.
  \item (iii) There exists an increasing sequence of stopping times $T_n$ on $W\time\Omega$ and a constant $0<\delta<1$ with the following property:
  $$\overline{P}\{(c,w)|\max_{T_n<t<T_{n+1}}d(\gamma_n(c,w)p,c(t))>k\}\leq\delta^k.$$
\end{itemize}

\begin{lemma}
  For almost any $(c,w)\in W\times \Omega$, we have
  \begin{equation}\label{limsup}
  \limsup_{n\rightarrow\infty}\frac{1}{n}\max_{T_n<t<T_{n+1}}d(\gamma_n(c,w)p,c(t))=0.
  \end{equation}
\end{lemma}

\begin{proof}
Denote by
$$A_{n,\varepsilon}=\{(c,w)|\max_{T_n<t<T_{n+1}}d(\gamma_n(c,w)p,c(t))>n\varepsilon\},$$
then we have
$$\sum_{n=0}^\infty\overline{P}(A_{n,\varepsilon})\leq \sum_{n=0}^\infty\delta^{n\varepsilon}<\infty.$$
By Borel-Cantelli,
$$\overline{P}(\limsup_{n\rightarrow\infty}A_{n,\varepsilon})=\overline{P}\left(\limsup_{n\rightarrow\infty}\left\{\frac{1}{n}\max_{T_n<t<T_{n+1}}d(\gamma_n(c,w)p,c(t))>\varepsilon\right\}\right)=0.$$
The lemma follows since this holds for any $\varepsilon>0.$
\end{proof}

The following result will be used in the proof of the Lemma \ref{brownianwalk}.
By Lemma \ref{finitemoment} and the definition of $\mu$, it is easy to see that all the assumptions of Lemma \ref{Guivarch} are satisfied,
thus (\ref{fracgammann}) also holds for our situation.

\begin{lemma}\cite{G}\label{Guivarch}
If $\Gamma$ is a finitely generated, non-amenable group, $\mu$ is a probability measure on $\Gamma$ with finite first moment and the support of $\mu$
generates $\Gamma$. Then there exists $A>0$, such that for $\mu^{\otimes\mathbb{N}}$-a.e. $\omega\in \Gamma^{\mathbb{N}}$,
the random sequence $\left\{ \gamma_k([\omega])\right\}$ in $\Gamma$
defined by (\ref{gammakomegaa}) satisfies
\begin{equation}\label{fracgammann}
\lim_{n\rightarrow\infty}\frac{|\gamma_n([\omega])|}{n}=A.
\end{equation}
\end{lemma}

\begin{lemma}\label{brownianwalk}
  The Brownian path converges at $\widetilde{M}(\infty)$ and its hitting measure at $\widetilde{M}(\infty)$
   is given by the hitting measure of the
 random walk on $\Gamma$ defined in Subsection \ref{randomwalk}.
\end{lemma}

\begin{proof}
By Lemma \ref{Guivarch}, we have
\begin{equation}\label{liminf}
  \liminf_{n\rightarrow\infty}\frac{1}{n}d(\gamma_n(c,w)p,p)>0.
\end{equation}
By (\ref{limsup}) and (\ref{liminf}), we obtain that for almost all $(c,w)\in W\times\Omega$,
$$d(p,c(t))\geq d(p,\gamma_n(c,w)p)-d(\gamma_n(c,w)p,c(t))\rightarrow\infty,(t\rightarrow\infty)$$
and
$$\max_{T_n<t<T_{n+1}}\measuredangle_p(c(t),\gamma_n(c,w)p)\rightarrow0,(n\rightarrow\infty).$$
Thus we have
$$\lim_{n\rightarrow\infty}\gamma_n(c,w)=\lim_{t\rightarrow\infty}c(t),a.s. (c,w)\in W\times\Omega.$$
We conclude by (i) and (ii) that the hitting measure at $\widetilde{M}(\infty)$ associated to $W\times\Omega$ exists
and is given by the hitting measure at $\widetilde{M}(\infty)$ of the $\mu$ random walk on $\Gamma.$
The lemma follows.
\end{proof}

Now we give the proof of Theorem \ref{Dirichlet}, that is,
for any continuous function $f$ on $\widetilde{M}(\infty),$ there exists a unique harmonic function on $\widetilde{M}$ which extends $f$ continuously.

\begin{proof} (Proof of Theorem \ref{Dirichlet})
Fix $p\in\widetilde{M}$, $\mu$ is the probability measure on $\Gamma$ given by $\mu(\alpha)=\mu_p(\alpha p).$
For any $q\in\widetilde{M},$ the hitting measure $\nu_q$ at $\widetilde{M}(\infty)$ of Brownian motion starting at $q$ exists.
For each $\alpha\in\Gamma$, $\nu_{\alpha q}=\alpha_*\nu_q$ since $\alpha$ acts as an isometry on $\widetilde{M}$,
where $\alpha_*\nu_q(\cdot)=\nu_q(\alpha^{-1}\cdot)$.

Define a function $h$ on $\widetilde{M}$  by $h(q)=\nu_q(f)=\int_{\xi\in\widetilde{M}(\infty)}f(\xi)\nu_q({\rm d}\xi)$,
then  $h$ is a harmonic function. For any $\alpha\in \Gamma,$
\begin{equation}\label{}
  h_f(\alpha)=\int_{\xi\in\widetilde{M}(\infty)}f(\alpha\xi)\nu({\rm d}\xi)=\int_{\eta\in\widetilde{M}(\infty)}f(\eta)(\alpha_{\ast}\nu)({\rm d}\eta)=h(\alpha p).
\end{equation}
Thus by Theorem \ref{hfgamman} and \ref{gammanp}, we have
\begin{equation}\label{alphanrightarrowxi}
h(\alpha_np)=h_f(\alpha_n)\rightarrow f(\xi), ~\mbox{if}~ \alpha_np\rightarrow \xi\in\widetilde{M}(\infty).
\end{equation}

We now prove that $h(q_n)\rightarrow f(\xi)$ if $q_n\rightarrow \xi\in\widetilde{M}(\infty)$.

Let $\xi\in\widetilde{M}(\infty).$ If there exists a sequence $\{q_n\}\subset \widetilde{M}$, such that
$q_n\rightarrow \xi$ but $h(q_n)\rightarrow f(\xi)+\varepsilon_0$ with a positive $\varepsilon_0>0.$
For the given $\xi$, let
$$\varepsilon_1=\sup\big\{\varepsilon_0|\mbox{there exists}~\{q_n\}\subset\widetilde{M},q_n\rightarrow\xi, \mbox{but}~h(q_n)\rightarrow f(\xi)+\varepsilon_0\big\}.$$
Now let $\{q_n\}$ be chosen so that $h(q_n)\rightarrow f(\xi)+\varepsilon_1$.

Since the balls $\{\alpha(B)\}_{\alpha\in\Gamma}$ is a cover of $\widetilde{M},$ there exists $\alpha_n\in\Gamma$
such that $q_n\in\alpha_n(B)$ for any $n$.
By the choices of $\varepsilon_1$ and $\{q_n\},$ we have $\alpha_nq\rightarrow\xi$ for any $q\in B,$ thus
$$\limsup_{n\rightarrow\infty}h(\alpha_nq)\leq f(\xi)+\varepsilon_1,\forall q\in B.$$

Without loss of generality, we assume that $f>0$ in the following. Choose a positive number $\varepsilon<\varepsilon_1,$
by the Harnack principle, there exists a sufficiently small positive number $\varrho$ such that
\begin{equation}\label{harnack}
\max_{y\in\alpha_n(B(p,\varrho R_B))}\tilde{h}(y)\leq \tilde{h}(\alpha_np)+\varepsilon
\end{equation}
for any positive harmonic function $\tilde{h}$ on $\alpha_n(B(p,2 R_B))$ such that $\min f\leq \tilde{h}\leq \max f,$
where $R_B$ is the radius of $B$.

\begin{figure}[hb]
		\centering
		\includegraphics[width=0.8\textwidth]{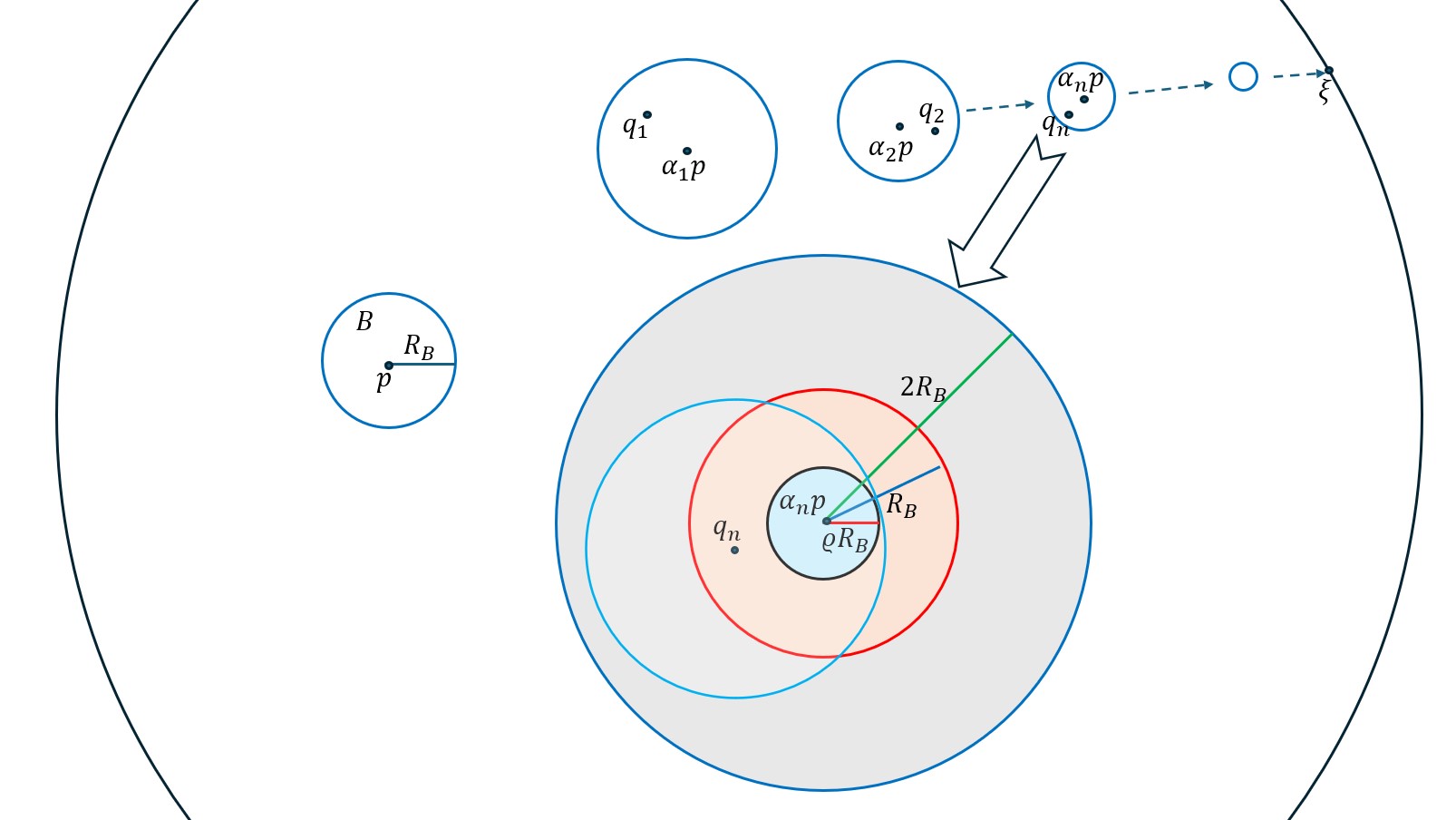}
\end{figure}

Denote by $\mathcal{E}_n$ the exit time from $\alpha_n(B(p,2 R_B))$ of the Brownian motion starting at $q_n$ and $\tau_n=\mathcal{E}_n\wedge1.$
Denote by $\lambda_n$ the probability measure on $\overline{\alpha_n(B(p,2 R_B))}$ induced by $\tau_n$. Then by the mean value theorem,
for any harmonic function
$\tilde{h}$ on $\overline{\alpha_n(B(p,2 R_B))}$, we have $\tilde{h}(q_n)=\lambda_n(\tilde{h}).$

Since all the $\alpha_n(B(p,2 R_B))$ are isometric and $q_n\in\frac12\alpha_n(B(p,2 R_B))$, there exists a constant $l>0$
such that $\lambda_n(\alpha_n(B(p,\varrho R_B))\geq l,\forall n.$

If there is a $\delta>0$ such that $\sup(h|_{\alpha_n(B(p,2 R_B))})\geq h(q_n)+\delta$ for any $n$.
Then there exists a sequence $\{q_n'\}$ such that $q_n'\rightarrow\xi$ but $h(q_n')\rightarrow f(\xi)+\varepsilon_1+\delta,$
which contradicts the choices of $\varepsilon_1$ and $\{q_n\}$. Thus we can assume that
$\sup(h|_{\alpha_n(B(p,2 R_B))})\leq h(q_n)+\frac1n$ (passing to a subsequence if necessary).
Then by (\ref{alphanrightarrowxi}) and (\ref{harnack}), we obtain that
\begin{eqnarray*}
h(q_n)&=&\lambda_n(h)=\int_{x\in\overline{\alpha_n(B(p,2 R_B))}}h(x)\lambda_n({\rm d}x)\\
&=&\int_{x\in\overline{\alpha_n(B(p,2 R_B))}-{\alpha_n(B(p,\varrho R_B))}}h(x)\lambda_n({\rm d}x)+\int_{x\in{\alpha_n(B(p,\varrho R))}}h(x)\lambda_n({\rm d}x)\\
&\leq&(h(q_n)+\frac1n)(1-\lambda_n(\alpha_n(B(p,\varrho R_B))))+(h(\alpha_np)+\varepsilon))\lambda_n(\alpha_n(B(p,\varrho R_B)))\\
&=&\lambda_n(\alpha_n(B(p,\varrho R_B)))[h(\alpha_np)-h(q_n)+\varepsilon]+h(q_n)+\frac1n-\frac1n\lambda_n(\alpha_n(B(p,\varrho R_B))).
\end{eqnarray*}
Let $n\rightarrow\infty$ in the above inequality, we have $0\leq l(-\varepsilon_1+\varepsilon),$ which contradicts $\varepsilon_1>\varepsilon.$

Similarly, we can also obtain that $\lim_{n\rightarrow\infty}h(q_n)\geq f(\xi)$ for any sequence which satisfies $q_n\rightarrow \xi$,
and the theorem follows.
\end{proof}

\section{\bf The Poisson Boundary of $(\Gamma,\mu)$}\label{question}
\setcounter{equation}{0}\setcounter{theorem}{0}

In this section, we consider the Poisson Boundary of $(\Gamma,\mu)$. We will establish an isometric isomorphism between the space of bounded $\mu$-harmonic functions on $(\Gamma,\mu)$ and $L^{\infty}(\widetilde{M}(\infty),\nu_{\ast})$, thus the space $(\widetilde{M}(\infty),\nu_{\ast})$ is the Poisson boundary of the pair $(\Gamma,\mu)$.
See \cite{Fu3, KV} for more information about Poisson boundary.

\subsection{Symbolic Dynamics}
Let $p\in \widetilde{M}$ be a fixed point, $\nu$ be a probability measure on the ideal boundary $\widetilde{M}(\infty)$,
and $\mu$ be a probability measure on $\Gamma$ such that its support generates $\Gamma$ as a semi-group.
Suppose that $\mu$ has $\mathbf{finite~first~moment}$, i.e.,
$$
\sum_{\alpha\in\Gamma}d(p,\alpha p)\mu(\alpha)<+\infty.
$$

Let $(\Gamma^{\mathbb{N}},\mu^{\otimes \mathbb{N}})$ be the product probability space, as in the theory of symbolic dynamics,
we define the $\mathbf{shift~transformation}$ as
\begin{displaymath}
			\begin{aligned}
				S:~& ~~~~~(\Gamma^{\mathbb{N}},\mu^{\otimes \mathbb{N}}) & \longrightarrow    &~~~~~~~~~(\Gamma^{\mathbb{N}},\mu^{\otimes \mathbb{N}}), \\
				       & [\omega]=(\omega_{0},\omega_{1},\omega_{2},...) &\longmapsto &~~~S[\omega]=(\omega_{1},\omega_{2},\omega_{3},...).
			\end{aligned}
\end{displaymath}
It's a classical result that the shift transformation $S$ is ergodic with respect to the measure $\mu^{\otimes \mathbb{N}}$.

For each $[\omega]=(\omega_{0},\omega_{1},\omega_{2},...)\in\Gamma^{\mathbb{N}}$, define
$$[\omega]^{(n)}\triangleq \omega_{0}\omega_{1}...\omega_{n-1},~~~~~n\in \mathbb{N}^{+}.$$

By Lemma \ref{Guivarch}, there exists a positive number $A>0$ which is independent of the point $p$, such that for
$\mu^{\otimes \mathbb{N}}$-a.e. $[\omega]=(\omega_{0},\omega_{1},\omega_{2},...)\in\Gamma^{\mathbb{N}}$,
\begin{equation}\label{eq79}
		\lim_{n\rightarrow +\infty}\frac{1}{n}\big([\omega]^{(n)}p,~p\big)=A.
\end{equation}

Let $[\omega]=(\omega_{0},\omega_{1},\omega_{2},...)\in\Gamma^{\mathbb{N}}$, denote by
\begin{equation}\label{eq79.1}
[\omega](\infty)\triangleq \lim_{n\rightarrow +\infty}[\omega]^{(n)}(p)
\end{equation}
if this limit exists.
Proposition \ref{pro20}(6) implies that the limit doesn't depend on the point $p$.
By Theorem \ref{gammanp} we know that for
$\mu^{\otimes \mathbb{N}}$-a.e. $[\omega]\in\Gamma^{\mathbb{N}}$ the limit in \eqref{eq79.1} exists.

Next we define a probability measure $\nu$ on the ideal boundary $\widetilde{M}(\infty)$ induced by $\mu$.
Let $B\subset \widetilde{M}(\infty)$ be a Borel subset, define
\begin{equation}\label{eq79.2}
\nu(B)\triangleq \mu^{\otimes \mathbb{N}}\Big\{ [\omega] \in \Gamma^{\mathbb{N}}~\big|~[\omega](\infty)\in B\Big\}.
\end{equation}
It's easy to see that $\nu$ is a probability measure on $\widetilde{M}(\infty)$ by Theorem \ref{gammanp}.

\begin{lemma}\label{le16}
The measure $\nu$ defined as in \eqref{eq79.2} has the following properties:

(1)  $\mu\ast\nu=\nu$.

(2)  For each $\alpha\in \Gamma$, the measures $\alpha_{\ast}\nu$ and $\nu$ are equivalent.

(3)  The action of $\Gamma$ on the ideal boundary $\widetilde{M}(\infty)$ is ergodic with respect to $\nu$.

(4)  The measure $\nu$ is positive on open subsets of $\widetilde{M}(\infty)$.
\end{lemma}
\begin{proof}
(1) By the Definition \ref{def10}, for any Borel subset $B\subset \widetilde{M}(\infty)$,
\begin{displaymath}
		\begin{aligned}
			(\mu\ast\nu)(B) & = \int_{\widetilde{M}(\infty)}\mathbf{1}_{B}(\xi)d(\mu\ast\nu)(\xi)                   \\
			              & = \sum_{\alpha\in\Gamma}\Big(\int_{\widetilde{M}(\infty)}\mathbf{1}_{B}(\alpha\xi)d\nu(\xi)\Big)\mu(\alpha) \\
			              & =  \sum_{\alpha\in\Gamma}\Big(\int_{\widetilde{M}(\infty)}\mathbf{1}_{\alpha^{-1}B}(\xi)d\nu(\xi)\Big)\mu(\alpha)    \\
                          & =  \sum_{\alpha\in\Gamma}\mu(\alpha)\nu(\alpha^{-1}B)           \\
                          & =  \sum_{\alpha\in\Gamma}\mu(\alpha)\mu^{\otimes \mathbb{N}}
                               \Big\{ [\omega]=(\omega_{0},\omega_{1},\omega_{2},...) \in \Gamma^{\mathbb{N}}~\big|~[\omega](\infty)\in \alpha^{-1}B\Big\} \\
                          & =  \sum_{\alpha\in\Gamma}\mu^{\otimes \mathbb{N}}
                                \Big\{ [\omega]=(\alpha,\omega_{0},\omega_{1},...) \in \Gamma^{\mathbb{N}}~\big|~[\omega](\infty)\in B\Big\} \\
                          & =  \mu^{\otimes \mathbb{N}}
                                \Big\{ [\omega]=(\omega_{0},\omega_{1},\omega_{2},...) \in \Gamma^{\mathbb{N}}~\big|~[\omega](\infty)\in B\Big\} \\
                          & =  \nu(B).\\
		\end{aligned}
\end{displaymath}

(2) We need to show that for each $\alpha\in \Gamma$, if a Borel subset $B\subset \widetilde{M}(\infty)$ with $\nu(B)>0$,
then $(\alpha_{\ast}\nu)(B)=\nu(\alpha^{-1}B)>0$, and vice versa.

Since $\mu$ is a probability measure on $\Gamma$ with the property that the support of $\mu$ generates $\Gamma$ as a semi-group,
we only need to show the above claim is valid for $\alpha \in \mathrm{Supp}(\mu)$, thus $\mu(\alpha)>0$ and $\mu(\alpha^{-1})>0$.

If $\nu(B)>0$, then
\begin{displaymath}
		\begin{aligned}
			(\alpha_{\ast}\nu)(B) ~& = ~         \nu(\alpha^{-1}B)                  \\
                                  ~& = ~         \mu^{\otimes \mathbb{N}}
                                               \Big\{ [\omega]=(\omega_{0},\omega_{1},\omega_{2},...) \in \Gamma^{\mathbb{N}}~\big|~[\omega](\infty)\in \alpha^{-1}B\Big\} \\
                                  ~& \geq~  \mu^{\otimes \mathbb{N}}
                                               \Big\{ [\omega]=(\alpha^{-1},\omega_{0},\omega_{1},...) \in \Gamma^{\mathbb{N}}~\big|~[\omega](\infty)\in \alpha^{-1}B\Big\} \\
                                  ~& = ~         \mu(\alpha^{-1})\mu^{\otimes \mathbb{N}}
                                               \Big\{ [\omega]=(\omega_{0},\omega_{1},\omega_{2},...) \in \Gamma^{\mathbb{N}}~\big|~[\omega](\infty)\in B\Big\} \\
                                  ~& =  ~        \mu(\alpha^{-1})\nu(B)\\
                                  ~& > ~  0.\\
		\end{aligned}
\end{displaymath}

(3) Since the limit in \eqref{eq79.1} $\mu^{\otimes \mathbb{N}}$-a.e. exists, we can define a map from a subset of $\Gamma^{\mathbb{N}}$ with full
$\mu^{\otimes \mathbb{N}}$-measure to $\widetilde{M}(\infty)$, for the sake of simplicity, we still denote this map as
\begin{displaymath}
			\begin{aligned}
				\mathcal{Q}:~& ~~~~~~~(\Gamma^{\mathbb{N}},\mu^{\otimes \mathbb{N}}) & \longrightarrow    &~~~(\widetilde{M}(\infty),\nu), \\
				       & [\omega]=(\omega_{0},\omega_{1},\omega_{2},...) &\longmapsto &~~~\mathcal{Q}[\omega]=[\omega](\infty).
			\end{aligned}
\end{displaymath}
This map $\mathcal{Q}$ is measure preserving. Let $B \subset \widetilde{M}(\infty)$ be a $\Gamma$-invariant subset,
then $\mathcal{Q}^{-1}(B)$ is a $S$-invariant subset of $\Gamma^{\mathbb{N}}$. By the ergodicity of the shift map $S$ with respect to the measure $\mu^{\otimes \mathbb{N}}$,
we know that $\mu^{\otimes \mathbb{N}}(\mathcal{Q}^{-1}(B))= 0$ or $1$, furthermore, $\nu(B)=\mu^{\otimes \mathbb{N}}(\mathcal{Q}^{-1}(B))= 0$ or $1$. Thus $\nu$ is ergodic.

(4) Let $B$ be an open subset of the ideal boundary $\widetilde{M}(\infty)$, for any point $\xi\in B$, by Proposition \ref{pro20}(5),
we know that $\overline{\Gamma(\xi)}=\widetilde{M}(\infty)$. Since $B$ is open, $\overline{\Gamma(\xi)}\subset \Gamma(B)$.
Therefore $\Gamma(B)=\widetilde{M}(\infty)$. If $\nu(B)=0$, by the second item of this Proposition, $\nu(\alpha B)=\big(\alpha_{\ast}^{-1}\nu\big)(B)=0$ for each $\alpha\in\Gamma$.
However, the discrete group $\Gamma$ is countable, we have
$$0=\sum_{\alpha\in\Gamma}\nu(\alpha B)\geqslant \nu\Big(\bigcup_{\alpha\in\Gamma} \alpha B\Big)=\nu\big(\widetilde{M}(\infty)\big)=1.$$
Therefore $\nu(B)> 0$.
\end{proof}

The probability measure $\mu$ induces a probability measure on $\Gamma$, denoted by $\check{\mu}$, defined as follows:
for any $\alpha\in\Gamma$, $\check{\mu}(\alpha)\triangleq\mu(\alpha^{-1})$.
It's easy to see that the support of $\check{\mu}$ generates $\Gamma$ as a semi-group
and $\check{\mu}$ has finite first moment, since $\mu$ does. Similar to \eqref{eq79.2}, one can define a probability measure
$\check{\nu}$ induced by $\check{\mu}$. Thus Lemma \ref{le16} is valid for $\check{\nu}$.

Now we consider the space of the two-sided sequences $(\Gamma^{\mathbb{Z}},\mu^{\otimes \mathbb{Z}})$, the shift map $T$ is defined as
\begin{displaymath}
			\begin{aligned}
				T:~& ~~~~~~~(\Gamma^{\mathbb{Z}},\mu^{\otimes \mathbb{Z}}) & \longrightarrow    &~~~~~~~~~~~(\Gamma^{\mathbb{Z}},\mu^{\otimes \mathbb{Z}}), \\
				       & [\omega]=(...,\omega_{-1},\omega_{0},\omega_{1},...) &\longmapsto &~~~T[\omega]=(...,\omega_{0},\omega_{1},\omega_{2},...).
			\end{aligned}
\end{displaymath}
The theory of symbolic dynamics tell us that the shift map $T$ is ergodic with respect to the measure $\mu^{\otimes \mathbb{Z}}$.

We can project a two-sided sequence in $(\Gamma^{\mathbb{Z}},\mu^{\otimes \mathbb{Z}})$ to two one-sided sequences
adapted to $(\Gamma^{\mathbb{N}},\mu^{\otimes \mathbb{N}})$ and $(\Gamma^{\mathbb{N}},\check{\mu}^{\otimes \mathbb{N}})$, respectively.
More precisely, the two projections are defined as following:
\begin{displaymath}
			\begin{aligned}
				\pi_{+}:~& ~~~~~~~(\Gamma^{\mathbb{Z}},\mu^{\otimes \mathbb{Z}}) & \longrightarrow    &~~~~~~~~~~~(\Gamma^{\mathbb{N}},\mu^{\otimes \mathbb{N}}), \\
				       & [\omega]=(...,\omega_{-1},\omega_{0},\omega_{1},...) &\longmapsto &~~~\pi_{+}[\omega]=(\omega_{0},\omega_{1},\omega_{2},...)\triangleq [\omega]_{+}.
			\end{aligned}
\end{displaymath}
\begin{displaymath}
			\begin{aligned}
				\pi_{-}:~& ~~~~~~~(\Gamma^{\mathbb{Z}},\mu^{\otimes \mathbb{Z}}) & \longrightarrow    &~~~~~~~~~~~(\Gamma^{\mathbb{N}},\check{\mu}^{\otimes \mathbb{N}}), \\
				       & [\omega]=(...,\omega_{-1},\omega_{0},\omega_{1},...) &\longmapsto &~~~\pi_{-}[\omega]=(\omega^{-1}_{-1},\omega^{-1}_{-2},\omega^{-1}_{-3},...)
                         \triangleq [\omega]_{-}.
			\end{aligned}
\end{displaymath}
It's easy to see the two maps $\pi_{+}$ and $\pi_{-}$ are measure preserving, thus the following two limits exist almost surely,
$$
[\omega](+\infty)\triangleq [\omega]_{+}(\infty),~~~[\omega](-\infty)\triangleq [\omega]_{-}(\infty).
$$
The hitting measures on the ideal boundary are $\nu$ and $\check{\nu}$, respectively.

Now there is a nature map from $\Gamma^{\mathbb{Z}}$ to $\widetilde{M}^{2}(\infty)$ induced by $\pi_{+}$ and $\pi_{-}$:
\begin{displaymath}
			\begin{aligned}
				\pi:~& \Gamma^{\mathbb{Z}} & \longrightarrow    &~~~\widetilde{M}^{2}(\infty), \\
				       & [\omega] &\longmapsto &~~~(\pi_{-}[\omega],\pi_{+}[\omega])= \big([\omega](-\infty),[\omega](+\infty)\big).
			\end{aligned}
\end{displaymath}
By the independence of $[\omega]_{-}$ and $[\omega]_{+}$, the distribution of $\pi([\omega])$ is $\check{\nu}\otimes\nu$.

For each $[\omega]=(...,\omega_{-1},\omega_{0},\omega_{1},...)\in \Gamma^{\mathbb{Z}}$,
\begin{displaymath}
			\begin{aligned}
	  \big(T[\omega]\big)(+\infty) ~& = ~\lim_{n\rightarrow +\infty}(\omega_{1}\omega_{2}...\omega_{n})(p)                  \\
                                   ~& = ~ \omega^{-1}_{0}\lim_{n\rightarrow +\infty}(\omega_{0}\omega_{1}...\omega_{n})(p) \\
                                   ~& = ~ \omega^{-1}_{0}[\omega](+\infty),\\
			\end{aligned}
\end{displaymath}
\begin{displaymath}
			\begin{aligned}
      \big(T[\omega]\big)(-\infty) ~& =~\lim_{n\rightarrow +\infty}(\omega^{-1}_{0}\omega^{-1}_{-1}...\omega^{-1}_{-n+1})(p)                  \\
                                   ~& = ~  \omega^{-1}_{0}\lim_{n\rightarrow +\infty}(\omega^{-1}_{-1}\omega^{-1}_{-2}...\omega^{-1}_{-n+1})(p) \\
                                   ~& = ~  \omega^{-1}_{0}[\omega](-\infty),\\
			\end{aligned}
\end{displaymath}
i.e.,
\begin{equation}\label{eq79.5}
T[\omega](+\infty)=\omega^{-1}_{0}[\omega](+\infty),~~~T[\omega](-\infty)=\omega^{-1}_{0}[\omega](-\infty).
\end{equation}

\begin{lemma}\label{le19}
The diagonal action of $\Gamma$ on $\widetilde{M}^{2}(\infty)$ is ergodic with respect to $\check{\nu}\otimes\nu$.
\end{lemma}
\begin{proof}
If $B \subset \widetilde{M}^{2}(\infty)$ is a $\Gamma$-invariant subset, by \eqref{eq79.5}
$\pi^{-1}(B)$ is a $T$-invariant subset of $\Gamma^{\mathbb{Z}}$. By the ergodicity of the shift map $T$ with respect to the measure $\mu^{\otimes \mathbb{Z}}$,
we have $\mu^{\otimes \mathbb{Z}}(\pi^{-1}(B))= 0$ or $1$. Since the map $\pi$ is measure preserving,
$\check{\nu}\otimes\nu(B)=\mu^{\otimes \mathbb{Z}}(\pi^{-1}(B))= 0$ or $1$.
\end{proof}

From now on, we assume that $\widetilde{M}$ is a rank $1$ simply connected manifolds without focal points.
We know that the set of rank $1$ vectors (see Definition \ref{def1}) is a open dense subset of $T^{1}\widetilde{M}$ (cf. \cite{BBE, LWW}).
Thus if we denote by $\mathfrak{R}$ as the set of endpoints at infinity of rank $1$ geodesics, i.e.,
$$
\mathfrak{R}\triangleq \big\{ (c(-\infty),c(+\infty))~|~c~is ~a ~rank ~1 ~geodesic ~in ~\widetilde{M} \big\},
$$
then $\mathfrak{R}$ is a open set in $\widetilde{M}^{2}(\infty)$.

\begin{lemma}\label{le23}
For $\check{\nu}\otimes\nu$-a.e. $(\xi,\eta)\in\widetilde{M}^{2}(\infty)$, there exists a rank $1$ (thus unique up to reparametrization) geodesic $c$
connecting $\xi$ and $\eta$.
\end{lemma}
\begin{proof}
As the isometries map rank $1$ geodesics to rank $1$ geodesics, we know that $\mathfrak{R}$ is a $\Gamma$-invariant subset of $\widetilde{M}^{2}(\infty)$.
By Lemma \ref{le16}(4), $\check{\nu}\otimes\nu(\mathfrak{R})>0$ since $\mathfrak{R}$ is open. Then Lemma \ref{le19} implies that $\check{\nu}\otimes\nu(\mathfrak{R})=1$.
Thus the conclusion is valid.
\end{proof}

\begin{lemma}\label{le25}
For each $R>0$,
$$
\lim_{n\rightarrow+\infty}\mu^{\otimes \mathbb{N}}\Big\{[\omega]\in \Gamma^{\mathbb{N}}~\big|~d([\omega]^{n}(p),c_{p,[\omega](\infty)}) > R \Big\}
=\mu^{\otimes \mathbb{Z}}\Big\{[\omega]\in \Gamma^{\mathbb{Z}}~\big|~d(p,c_{[\omega](-\infty),[\omega](+\infty)}) > R \Big\}.
$$
\end{lemma}
\begin{proof}
Let $[\omega]=(\omega_{0},\omega_{1},\omega_{2},...) \in\Gamma^{\mathbb{N}}$,
\begin{displaymath}
			\begin{aligned}
      d([\omega]^{n}(p),c_{p,[\omega](\infty)}) ~& =~d(\omega_{0}\omega_{1}...\omega_{n-1}p,c_{p,[\omega](\infty)}) \\
                                                ~& =~d(p,c_{\omega^{-1}_{n-1}\omega^{-1}_{n-2}...\omega^{-1}_{0}p,
                                                      ~\omega^{-1}_{n-1}\omega^{-1}_{n-2}...\omega^{-1}_{0}[\omega](\infty)}) \\
                                                ~& =~d(p,c_{\omega^{-1}_{n-1}\omega^{-1}_{n-2}...\omega^{-1}_{0}p,
                                                      ~\lim_{m\rightarrow+\infty}\omega_{n}\omega_{n+1}...\omega_{n+m}p}).   \\
			\end{aligned}
\end{displaymath}
By the independence of $\omega_{n},\omega_{n+1},...$ and $\omega_{0},\omega_{1},...,\omega_{n-1}$, we know that the distribution of
$d([\omega]^{n}(p),c_{p,[\omega](\infty)})$ is the same as the distribution of
$d(p,c_{\omega^{-1}_{n-1}\omega^{-1}_{n-2}...\omega^{-1}_{0}p,~\omega^{-1}_{n-1}\omega^{-1}_{n-2}...\omega^{-1}_{0}[\omega](\infty)})$.
Suppose $\lim_{n\rightarrow+\infty}\omega^{-1}_{n-1}\omega^{-1}_{n-2}...\omega^{-1}_{0}p=\xi\in \widetilde{M}(\infty)$, and
$\lim_{n,m\rightarrow+\infty}\omega_{n}\omega_{n+1}...\omega_{n+m}p=\eta\in \widetilde{M}(\infty)$. By Lemma \ref{lem6}, we know
$c_{\omega^{-1}_{n-1}\omega^{-1}_{n-2}...\omega^{-1}_{0}p,~\lim_{m\rightarrow+\infty}\omega_{n}\omega_{n+1}...\omega_{n+m}p}\rightarrow c_{\xi,\eta}$.
Thus the conclusion is valid.
\end{proof}

The Lemma \ref{le23} and Lemma \ref{le25} implies the following result.
\begin{corollary}\label{co27}
Fix a positive constant $\delta$, if the number $R>0$ is large enough, then for $n$ sufficiently large, we have
$\mu^{\otimes \mathbb{N}}\big\{[\omega]\in \Gamma^{\mathbb{N}}~\big|~d([\omega]^{n}(p),c_{p,[\omega](\infty)}) > R \big\} \leq \delta$.
\end{corollary}

\subsection{Isomorphism Between $\mu$-harmonic Function Space and $L^\infty(\widetilde{M}(\infty),\nu)$}

Given $q\in\widetilde{M}$, we can define a probability measure $\mu_q$ on $\Gamma$ by Definition \ref{defmup}.
By Theorem \ref{gammanp},  for almost any (in the sense of $\mu_q^{\otimes\mathbb{N}}$) sequence $[\omega]=(\omega_0,\omega_1,\omega_2,\cdots)$
in $\Gamma$ we have $[\omega]^{(n)}p_0$ ($\forall p_0\in \widetilde{M}$) tends to a limit in $\widetilde{M}(\infty)$ and
the hitting probability $\nu_q$ satisfies $\mu_q\ast\nu_q=\nu_q.$
By Lemma \ref{brownianwalk}, the hitting measure at $\widetilde{M}(\infty)$ of the Brownian motion starting at $q$
is just equal to $\nu_q$. Since $q\in \widetilde{M}$ is arbitrary, all the measures $\nu_q,q\in\widetilde{M}$ are equivalent by the
Harnack inequality and they define the harmonic measure class $\nu_\ast$ on $\widetilde{M}(\infty).$

In \cite{BL} it is shown that the properties of random walks on $\Gamma$  also hold for Brownian motion on $\widetilde{M}$,
 thus to obtain the isomorphism between $\mathcal{H}^{\infty}(\widetilde{M})$ and $L^{\infty}(\widetilde{M}(\infty),\nu_{\ast})$, we only need to
 prove the isomorphism between the space of bounded $\mu_q$-harmonic functions on $\Gamma$ and $L^\infty(\widetilde{M}(\infty),\nu_q)$ for any $q\in\widetilde{M}$.

Choose a fixed point $p\in \widetilde{M}$. Let $\mu=\mu_p$ be the probability measure on $\Gamma$ defined in Definition \ref{defmup}
and $\nu=\nu_p$ be the probability measure on $\widetilde{M}(\infty)$, we know that (\ref{eq67}) is satisfied for $\mu$ and $\nu$.
In the following we will prove that the space of bounded $\mu$-harmonic functions on $\Gamma$ and $L^\infty(\widetilde{M}(\infty),\nu)$ are
isomorphic.

Two sequences $[\omega],[\tilde{\omega}]\in\Gamma^{\mathbb{N}}$ are called equivalent if for all sufficiently large $n$, $[\omega]^{(n)}=[\tilde{\omega}]^{(n)}$ .
Denote by $L_\varepsilon^\infty$ the space of bounded measurable functions on $\Gamma^{\mathbb{N}}$ which are constant on equivalent classes.

\begin{lemma}\label{nuharisolvar}
The space of bounded $\mu$-harmonic functions on $\Gamma$ is canonically isomorphic to the space $L_\varepsilon^\infty.$
\end{lemma}

\begin{proof}
If $h$ is a bounded $\mu$-harmonic function,  by Lemma \ref{martingale}, $\{h([\omega]^{(n)})\}$ is a martingale. By the
martingale convergence theorem, $\lim_{n\rightarrow\infty}h([\omega]^{(n)})$ exist for $\mu^{\otimes\mathbb{N}}$-a.e. $[\omega]\in\Gamma^{\mathbb{N}}$,
let
$$f([\omega])=\lim_{n\rightarrow\infty}h([\omega]^{(n)}),$$
we get $f\in L_\varepsilon^\infty$.

For any $\alpha\in \Gamma$, let $\alpha[\omega]=\alpha(\omega_0,\omega_1,\cdots)=(\alpha,\omega_0,\omega_1,\cdots).$
If $f\in L_\varepsilon^\infty$, let
$$h(\alpha)=\sum_{[\omega]\in\Gamma^{\mathbb{N}}}f(\alpha[\omega])\mu^{\otimes\mathbb{N}}({\rm d}[\omega]),\forall \alpha\in\Gamma,$$
we get a bounded $\mu$-harmonic function $h$ since
\begin{eqnarray*}
\sum_{\alpha'\in\Gamma}h(\alpha\alpha')\mu(\alpha')&=&\sum_{\alpha'\in\Gamma}\sum_{[\omega]\in\Gamma^{\mathbb{N}}}f(\alpha\alpha'[\omega])
\mu^{\otimes\mathbb{N}}({\rm d}[\omega])\mu(\alpha')
\\
&=&\sum_{[\omega']\in\Gamma^{\mathbb{N}}}f(\alpha[\omega'])\mu^{\otimes\mathbb{N}}({\rm d}[\omega'])\\
&=&h(\alpha),\forall \alpha\in \Gamma.
\end{eqnarray*}
\end{proof}

Define the exit map $\mathcal{E}$ from $\Gamma^{\mathbb{N}}$ to $\widetilde{M}(\infty)$ by
$$\mathcal{E}([\omega])=[\omega](\infty)=\lim_{n\rightarrow +\infty}[\omega]^{(n)}(p).$$
Then $\mathcal{E}$ is constant on equivalence classes, thus
\begin{equation}\label{isomorphism}
f\rightarrow f\circ \mathcal{E}
\end{equation}
gives an isomorphism between $L^\infty(\widetilde{M}(\infty),\nu)$ and a subspace of $L_\varepsilon^\infty$.
For the corresponding bounded $\mu$-harmonic function $h_f$ we have
$$\lim_{n\rightarrow\infty}h_f([\omega]^{(n)})=f([\omega](\infty)).$$
To show that the above subspace is the whole space $L_\varepsilon^\infty$, we need the estimations on entropies.

\begin{Defi}\label{entropy}
The entropy $H(\mu)$ of the measure $\mu$ on $\Gamma$ is defined by
$$H(\mu)=-\sum_{\alpha\in\Gamma}\mu(\alpha)\ln \mu(\alpha).$$
The entropy $\beta(\widetilde{M}(\infty),\nu)$ of $(\widetilde{M}(\infty),\nu)$ is defined by
\begin{eqnarray*}
\beta(\widetilde{M}(\infty),\nu)&=&-\sum_{\alpha\in\Gamma}\left(\int_{\xi\in\widetilde{M}(\infty)}
\left(\ln\frac{{\rm d}((\alpha^{-1})_*\nu)(\xi)}{{\rm d}\nu(\xi)}\right){\rm d}\nu(\xi)\right)\mu(\alpha)\\
&=&
\sum_{\alpha\in\Gamma}\left(\int_{\xi\in\widetilde{M}(\infty)}
\left(\ln\frac{{\rm d}(\alpha_*\nu)(\xi)}{{\rm d}\nu(\xi)}\right){\rm d}(\alpha_*\nu)(\xi)\right)\mu(\alpha).
\end{eqnarray*}
\end{Defi}

The entropy $H(\mu)$ of $\mu$ defined in Definition \ref{entropy} is finite. In fact, since $\widetilde{M}/\Gamma$ is compact, the curvature of $\widetilde{M}$ is bounded, there exists a constant $C>0$ such that
$|\Gamma_n|\leq C{\rm e}^{Cn}$, where $\Gamma_n=\{\alpha\in\Gamma|n-1\leq d(p,\alpha p)<n\}$, $p$ is some fixed point in $\widetilde{M}$.
By the Jensen inequality, we have
\begin{eqnarray}
H(\mu)&=&-\sum_{\alpha\in\Gamma}\mu(\alpha)\ln \mu(\alpha)
=\sum_{n=1}^\infty\left(-\sum_{\alpha\in\Gamma_n}\mu(\alpha)\ln \mu(\alpha)\right)\nonumber\\
&\leq&\sum_{n=1}^\infty\left(-\sum_{\alpha\in\Gamma_n}\mu(\alpha)\ln \frac{\mu(\Gamma_n)}{|\Gamma_n|}\right)\nonumber\\
&\leq&\sum_{n=1}^\infty\mu(\Gamma_n)(Cn+\ln C)+\sum_{n=1}^\infty-\mu(\Gamma_n)\ln\mu(\Gamma_n)<+\infty,
\end{eqnarray}
since by Lemma \ref{finitemoment}, $\mu$ has finite first moment and  by using series theory,
for a sequence $\{u_n\}$ of non-negative numbers, $\sum_{n=1}^\infty nu_n<+\infty$ implies $-\sum_{n=1}^\infty u_n\ln u_n<+\infty.$

Define the probability measure $\mu^k$ on $\Gamma$ inductively by
$$\mu^k(\alpha)=\sum_{\alpha'\in\Gamma}\mu^{k-1}(\alpha')\mu(\alpha'^{-1}\alpha), k=1,2,\cdots, \mu_0=\mu.$$
That is, $\mu^k$ is the $k$-fold convolution of $\mu$ on $\Gamma$.
Then we have  (cf. \cite{KV})
\begin{equation*}
H(\mu^{n+m})\leq H(\mu^n)+H(\mu^m).
\end{equation*}

For a sequence $\{u_n\}$ of non-negative numbers, $u_{n+m}\leq u_n+u_m$ implies $\left\{\frac{u_n}{n}\right\}$ is convergent,
we obtained that the limit $\lim_{n\rightarrow\infty}\frac{1}{n}H(\mu^n)$ exists and is finite.

\begin{Defi}
Define the (Avez) entropy of the random walk on $\Gamma$ associated to $\mu$ by
\begin{equation*}
\beta(\mu)=\lim_{n\rightarrow\infty}\frac{1}{n}H(\mu^n).
\end{equation*}
\end{Defi}

By Lemma \ref{nuharisolvar}, to obtain Theorem \ref{Poisson}, we only need to prove the isomorphism between $L_\varepsilon^\infty$ and $L^\infty(\widetilde{M}(\infty),\nu)$. It has been proven in \cite{KV} that $\beta(\widetilde{M}(\infty),\nu)\leq \beta(\mu)$
and equality holds if and only if (\ref{isomorphism})
defines an isomorphism between $L^\infty(\widetilde{M}(\infty),\nu)$ and  $L_\varepsilon^\infty$.
Therefore we only need to prove $\beta(\widetilde{M}(\infty),\nu)\geq \beta(\mu)$.
For this purpose, we show another formula for $\beta(\widetilde{M}(\infty),\nu)$, that is
\begin{equation}\label{entropy2}
\beta(\widetilde{M}(\infty),\nu)=\lim_{n\rightarrow\infty}-\frac{1}{n}\ln\frac{{\rm d}((\omega_0^{-1}\omega_1^{-1}\cdots\omega_{n-1}^{-1})_*\nu)}
{{\rm d}\nu}(\xi),{a.s.}
[\omega]\in\Gamma^{\mathbb{N}},\xi\in \widetilde{M}(\infty).
\end{equation}

In fact, define $\tilde{S}:\Gamma^{\mathbb{N}}\times\widetilde{M}(\infty)\rightarrow\Gamma^{\mathbb{N}}\times\widetilde{M}(\infty)$
by $\tilde{S}([\omega],\xi)=(S[\omega],\omega_0\xi),$ then $\tilde{S}^i([\omega],\xi)=(S^i[\omega],\omega_{i-1}\cdots\omega_0\xi)$ and
$\tilde{S}$ preserves and is ergodic with respect to $\mu^{\otimes\mathbb{N}}\otimes\nu$ (cf. \cite{BL} V.2, \cite{B} $\S$7.3).
Let $F([\omega],\xi)=-\ln\frac{{\rm d}((\omega_0^{-1})_*\nu)}{{\rm d}\nu}(\xi)$, then by the ergodic theorem, we know that for $\mu^{\otimes\mathbb{N}}$-a.e. $[\omega]\in\Gamma^{\mathbb{N}}$ and $\xi\in\widetilde{M}(\infty),$
\begin{eqnarray*}
\beta(\widetilde{M}(\infty),\nu)&=&-\sum_{\alpha\in\Gamma}\left(\int_{\xi\in\widetilde{M}(\infty)}
\left(\ln\frac{{\rm d}((\alpha^{-1})_*\nu)(\xi)}{{\rm d}\nu(\xi)}\right){\rm d}\nu(\xi)\right)\mu(\alpha)\\
&=&\lim_{n\rightarrow\infty}\frac{1}{n}
\sum_{i=0}^{n-1}F(\tilde{S}^i([\omega],\xi))=
\lim_{n\rightarrow\infty}-\frac{1}{n}
\sum_{i=0}^{n-1}
\ln\frac{{\rm d}((\omega_i^{-1})_*\nu)}{{\rm d}\nu}(\omega_{i-1}\cdots\omega_0\xi)\\
&=&
\lim_{n\rightarrow\infty}-\frac{1}{n}
\sum_{i=0}^{n-1}
\ln\frac{{\rm d}((\omega_0^{-1}\cdots\omega_i^{-1})_*\nu)}{{\rm d}((\omega_0^{-1}\cdots\omega_{i-1}^{-1})_*\nu)}(\xi)\\
&=&
\lim_{n\rightarrow\infty}-\frac{1}{n}
\ln\frac{{\rm d}((\omega_0^{-1}\cdots\omega_{n-1}^{-1})_*\nu)}{{\rm d}\nu}(\xi).
\end{eqnarray*}

\begin{Defi}
Let $D>0$ and $p\in\widetilde{M}$ be a fixed point. Define two functions on $\widetilde{M}^2(\infty)$ as follows:
For $\xi,\eta\in\widetilde{M}(\infty)$,
\begin{eqnarray*}
t(\xi,\eta)&=&t_{p,D}(\xi,\eta)=\inf\{t>0|d(\sigma_{p,\xi}(t),\sigma_{p,\eta}(t))>D\},\\
d(\xi,\eta)&=&d_{p,D}(\xi,\eta)={\rm e}^{-t_{p,D}(\xi,\eta)}.
\end{eqnarray*}
Let
$$B(\xi,r)=B_{p,D}(\xi,r)=\{\eta\in\widetilde{M}(\infty)|d_{p,D}(\xi,\eta)\leq r\},$$
which will be called the ball of radius $r$ about $\xi$ although $d_{p,D}$ is not a metric in general.
\end{Defi}

Let
$$\mathfrak{R}_D=\left\{(z,\xi,\eta)\in\widetilde{M}^3(\infty)\left|\begin{array}{l}
(z,\xi)\in\mathfrak{R},(z,\eta)\in\mathfrak{R},\mbox{and}\\
\lim_{t\rightarrow-\infty}d(\sigma_{z,\xi}(t),\sigma_{z,\eta}(t))<\frac{D}{8}
\end{array}\right.\right\}.$$
For a given $\delta>0$, by Lemma \ref{le23}, we can choose a sufficiently large $D>0$ so that $\mathfrak{R}_D$ has measure greater than $1-\delta$ with respect to $\check{\nu}\otimes\nu\otimes\nu$.

By Theorem \ref{gammanp}, the following limit exists for $\mu^{\otimes\mathbb{N}}$-a.e. $[\omega]\in\Gamma^{\mathbb{N}}$,
\begin{equation}
z([\omega])\triangleq\lim_{n\rightarrow\infty}\omega_0^{-1}\omega_1^{-1}\cdots\omega_{n-1}^{-1}p.
\end{equation}
Let
$$\mathcal{P}_D=\left\{([\omega],\xi,\eta)\in\Gamma^{\mathbb{N}}\times(\widetilde{M}(\infty))^2\left|(z([\omega]),\xi,\eta)\in \mathfrak{R}_D\right.\right\}.$$
Since the distribution of $z([\omega])$ is $\check{\nu}$,
$\mathcal{P}_D$ has measure at least $1-\delta$ with respect to $\mu^{\otimes\mathbb{N}}\otimes\nu\otimes\nu$.

By Lemma \ref{Guivarch}, for $\mu^{\otimes\mathbb{N}}$-a.e. $[\omega]\in\Gamma^{\mathbb{N}}$,
\begin{equation}\label{Guivarch2}
\lim_{n\rightarrow\infty}\frac{1}{n}d(p,\omega_0^{-1}\omega_1^{-1}\cdots\omega_{n-1}^{-1}p)=A.
\end{equation}
Thus the set $\mathcal{P}_D^1$ of $([\omega],\xi,\eta)$ in $\mathcal{P}_D$ such that (\ref{Guivarch2}) holds
has measure at least $1-\delta$.

\begin{lemma}\label{-1ndpd}
For any $\delta>0$, there exists a sufficiently large $D>0$, such that the measure of the set
\begin{equation}\label{frac1nlndpd}
\left\{([\omega],\xi,\eta)\left|\lim_{n\rightarrow\infty}\frac{1}{n} t_{p,D}(\omega_{n-1}\cdots\omega_{0}\xi,\omega_{n-1}\cdots\omega_{0}\eta)=A\right.\right\}
\end{equation}
is at least $1-\delta$ with respect to $\mu^{\otimes\mathbb{N}}\otimes\nu\otimes\nu$.
\end{lemma}

\begin{proof}
We only need to prove that for sufficiently large $D>0$,
$$\lim_{n\rightarrow\infty}\frac{1}{n} t_{p,D}(\omega_{n-1}\cdots\omega_{0}\xi,\omega_{n-1}\cdots\omega_{0}\eta)=A$$
 holds for all $([\omega],\xi,\eta)\in\mathcal{P}_D^1$,
 since the set $\mathcal{P}_D^1$ of $([\omega],\xi,\eta)$ has measure at least $1-\delta$ for sufficiently large $D$.
  Now let $([\omega],\xi,\eta)\in\mathcal{P}_D^1$, and $z=z([\omega])$.

We can parameterize $c_{z,\xi}$ and $c_{z,\eta}$ proportionally to arc length such that for each $t$,
$c_{z,\xi}(t)$ and $c_{z,\eta}(t)$ are on the same horosphere about $z$. Then
\begin{equation}
\lim_{t\rightarrow-\infty}d(c_{z,\xi}(t),c_{z,\eta}(t))<\frac{D}{8},~~\lim_{t\rightarrow+\infty}d(c_{z,\xi}(t),c_{z,\eta}(t))=+\infty.
\end{equation}
Since $\frac{D}{8}<\frac{D}{4}<\frac{7D}{4}<\infty,$ we can choose times $t_-<t_+$ such that
\begin{equation}\label{fracd35d3}
d(c_{z,\xi}(t_-),c_{z,\eta}(t_-))=\frac{D}{4},~~d(c_{z,\xi}(t_+),c_{z,\eta}(t_+))=\frac{7D}{4}.
\end{equation}

\begin{figure}[hb]
		\centering
		\includegraphics[width=0.8\textwidth]{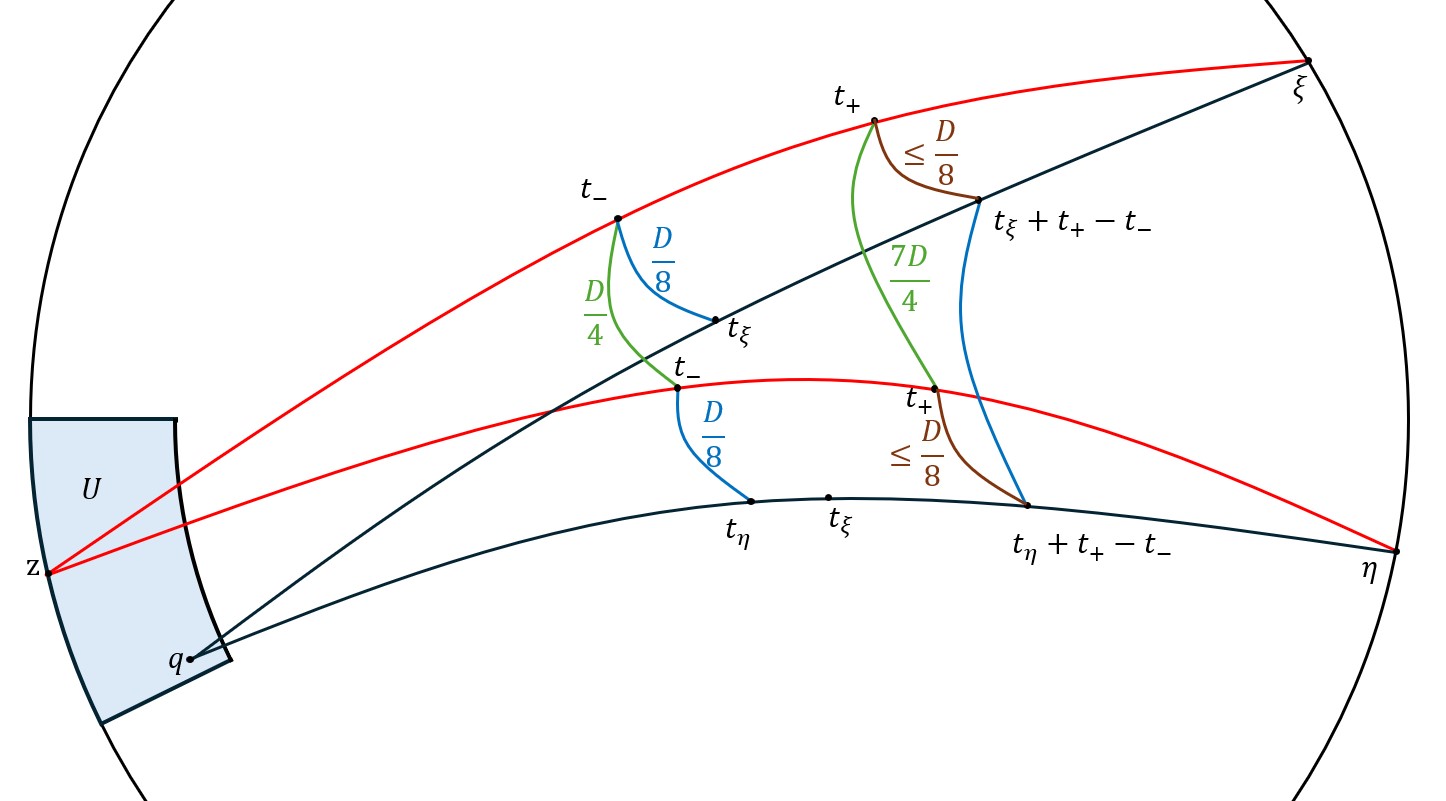}
\end{figure}

Since the geodesics $c_{z,\xi}$ and $c_{z,\eta}$ are regular, by Proposition 4 in \cite{OS} and Lemma~\ref{lem6},
there is a neighborhood $U$ of $z$ in $\overline{\widetilde{M}}$
such that
\begin{equation}
d(c_{z,\xi}(t),c_{q,\xi})\leq \frac{D}{8},~~d(c_{z,\eta}(t),c_{q,\eta})\leq \frac{D}{8},~~\forall t\geq t_-,q\in U.
\end{equation}
Now let $q\in U\cap\widetilde{M}$ and parameterize $c_{q,\xi},c_{q,\eta}$ proportionally to arc length such that
$c_{q,\xi}(0)=q=c_{q,\eta}(0).$ Then we can choose $t_\xi,t_\eta$ such that
\begin{equation}\label{leqfracd12}
d(c_{z,\xi}(t_-),c_{q,\xi}(t_\xi))= \frac{D}{8},~~d(c_{z,\eta}(t_-),c_{q,\eta}(t_\eta))= \frac{D}{8}.
\end{equation}
Thus we have
\begin{eqnarray}
&&d(c_{q,\xi}(t_\xi),c_{q,\eta}(t_\eta))\nonumber\\
&\leq&d(c_{q,\xi}(t_\xi),c_{z,\xi}(t_-))+d(c_{z,\xi}(t_-),c_{z,\eta}(t_-))+d(c_{z,\eta}(t_-),c_{q,\eta}(t_\eta))\nonumber\\
&\leq&\frac{D}{8}+\frac{D}{4}+\frac{D}{8}=\frac{D}{2},
\end{eqnarray}
and
\begin{equation}
|t_\xi-t_\eta|=|d(q,c_{q,\xi}(t_\xi))-d(q,c_{q,\eta}(t_\eta))|
\leq d(c_{q,\xi}(t_\xi),c_{q,\eta}(t_\eta))\leq\frac{D}{2}.
\end{equation}
Therefore,
\begin{equation}\label{leqd}
d(c_{q,\xi}(t_\xi),c_{q,\eta}(t_\xi))\leq|t_\xi-t_\eta|+d(c_{q,\xi}(t_\xi),c_{q,\eta}(t_\eta))\leq D.
\end{equation}

Since $c_{q,\xi}$ is asymptotic to $c_{z,\xi}$ and $c_{q,\eta}$ is asymptotic to $c_{z,\eta}$,
using (\ref{leqfracd12}) we get
\begin{equation}
d(c_{z,\xi}(t_+),c_{q,\xi}(t_\xi+t_+-t_-))\leq \frac{D}{8},~~d(c_{z,\eta}(t_+),c_{q,\eta}(t_\eta+t_+-t_-))\leq \frac{D}{8}.
\end{equation}
Therefore, by (\ref{fracd35d3}) we have
\begin{eqnarray}\label{geqd}
&&d(c_{q,\xi}(t_\xi+t_+-t_-),c_{q,\eta}(t_\xi+t_+-t_-))\nonumber\\
&\geq&d(c_{q,\xi}(t_\xi+t_+-t_-),c_{q,\eta}(t_\eta+t_+-t_-))-|t_\xi-t_\eta|\nonumber\\
&\geq&d(c_{z,\xi}(t_+),c_{z,\eta}(t_+))-
d(c_{z,\xi}(t_+),c_{q,\xi}(t_\xi+t_+-t_-))\nonumber\\
&&~~-d(c_{z,\eta}(t_+),c_{q,\eta}(t_\eta+t_+-t_-))-|t_\xi-t_\eta|\nonumber\\
&\geq&\frac{7D}{4}-\frac{D}{8}-\frac{D}{8}-\frac{D}{2}=D.
\end{eqnarray}

By (\ref{leqd}), (\ref{geqd}) and the definition of $t_{q,D}(\xi,\eta)$, we get
\begin{equation}
t_\xi\leq t_{q,D}(\xi,\eta)\leq t_\xi+t_+-t_-.
\end{equation}
Therefore we obtain
\begin{eqnarray}
&&d(q,c_{z,\xi}(t_-))-\frac{D}{8}\leq t_\xi\leq t_{q,D}(\xi,\eta)\leq t_\xi+t_+-t_-\nonumber\\
&\leq& d(q,c_{z,\xi}(t_-)+\frac{D}{8}+t_+-t_-,
\end{eqnarray}
that is
\begin{equation} \label{tqd}
|d(q,c_{z,\xi}(t_-))-t_{q,D}(\xi,\eta)|\leq\frac{D}{8}+t_+-t_-.
\end{equation}

Let $q_n=\omega_0^{-1}\cdots\omega_{n-1}^{-1}p$. Since
$z=\lim_{n\rightarrow\infty}\omega_0^{-1}\omega_1^{-1}\cdots\omega_{n-1}^{-1}p$,
it holds $q_n\in U$ for $n$ sufficiently large, thus
\begin{eqnarray*}
&&-|d(c_{z,\xi}(t_-),q_n)-t_{q_n,D}(\xi,\eta)|-d(p,c_{z,\xi}(t_-))\\
&\leq&d(c_{z,\xi}(t_-),q_n)-t_{q_n,D}(\xi,\eta)-d(p,c_{z,\xi}(t_-))\\
&\leq&d(p,q_n)-t_{q_n,D}(\xi,\eta)=d(p,q_n)-t_{p,D}(\omega_{n-1}\cdots\omega_0\xi,\omega_{n-1}\cdots\omega_0\eta)\\
&\leq&d(c_{z,\xi}(t_-),q_n)-t_{q_n,D}(\xi,\eta)+d(p,c_{z,\xi}(t_-))\\
&\leq&|d(c_{z,\xi}(t_-),q_n)-t_{q_n,D}(\xi,\eta)|+d(p,c_{z,\xi}(t_-)).
\end{eqnarray*}
By the above inequality and (\ref{tqd}), we can obtain that
\begin{equation}
|t_{p,D}(\omega_{n-1}\cdots\omega_0\xi,\omega_{n-1}\cdots\omega_0\eta)-d(p,q_n)|\leq C
\end{equation}
for a constant $C$ independent of $n$. That is
\begin{equation}
\left|\frac{t_{p,D}(\omega_{n-1}\cdots\omega_0\xi,\omega_{n-1}\cdots\omega_0\eta)}{n}-\frac{d(p,\omega_0^{-1}\cdots\omega_{n-1}^{-1}p )}{n}\right|\leq \frac{C}{n}.
\end{equation}
Let $n\rightarrow \infty$, the lemma follows from (\ref{Guivarch2}).
\end{proof}

Let $\delta>0, \beta=\beta({\widetilde{M}(\infty),\nu})$. Set
\begin{eqnarray}
\Sigma_n&=&\left\{\xi\in\widetilde{M}(\infty)\left|\nu\left(B(\xi,{\rm e}^{-n(A-\delta)})\right)\geq {\rm e}^{-n(\beta+\delta)}\right.\right\}\nonumber\\
&=&\left\{\xi\in\widetilde{M}(\infty)\left|\nu\left\{\eta\in\widetilde{M}(\infty)\left|t_{p,D}(\xi,\eta)\geq n(A-\delta)\right.\right\}\geq {\rm e}^{-n(\beta+\delta)}\right.\right\}
\end{eqnarray}
and
\begin{equation}
\tilde{\Sigma}_n=\left\{([\omega],\xi)\in\Gamma^{\mathbb{N}}\times\widetilde{M}(\infty)|\omega_{n-1}\cdots\omega_0\xi\in\Sigma_n\right\}.
\end{equation}
Set
\begin{eqnarray}
\Sigma_{([\omega],\xi)}^n&=&\left\{\eta\in\widetilde{M}(\infty)\left|\begin{array}{l}d_{p,D}(\omega_{n-1}\cdots\omega_0\xi,
\omega_{n-1}\cdots\omega_0\eta)\leq{\rm e}^{-n(A-\delta)},\\
\mbox{and}~\frac{{\rm d}((\omega_0^{-1}\cdots\omega_{n-1}^{-1})_*\nu)}{{\rm d}\nu}(\eta)\geq{\rm e}^{-n(\beta+\frac{\delta}{2})}.\end{array}\right.\right\}\nonumber\\
&=&\left\{\eta\in\widetilde{M}(\infty)\left|\begin{array}{l}t_{p,D}(\omega_{n-1}\cdots\omega_0\xi,
\omega_{n-1}\cdots\omega_0\eta)\geq{n(A-\delta)},\\
\mbox{and}~-\frac1n\ln\frac{{\rm d}((\omega_0^{-1}\cdots\omega_{n-1}^{-1})_*\nu)}{{\rm d}\nu}(\eta)\leq{(\beta+\frac{\delta}{2})}.\end{array}\right.\right\}
\end{eqnarray}
and
\begin{equation}
\hat{\Sigma}_n=\left\{([\omega],\xi)\in\Gamma^{\mathbb{N}}\times\widetilde{M}(\infty)|\nu(\Sigma_{([\omega],\xi)}^n)>\frac12\right\}.
\end{equation}

\begin{lemma}\label{sigman}
For $D$ sufficiently large, the measure of $\Sigma_n$ is at least $1-\delta$ with respect to $\nu$ for $n$  large enough.
\end{lemma}

\begin{proof}
By (\ref{entropy2}) and Lemma \ref{-1ndpd}, if $D$ is sufficiently large, the set $\hat{\Sigma}_n$ for all $n$ sufficiently large has measure at least $1-\delta$.
Choose $n$ large enough such that ${\rm e}^{-\frac{n\delta}{2}}<\frac12.$ Given $([\omega],\xi)\in\hat{\Sigma}_n,$
then for any $\eta\in\Sigma_{([\omega],\xi)}^n$, we have
\begin{eqnarray*}
{\rm d}\nu(\omega_{n-1}\cdots\omega_0\eta)={\rm d}((\omega_0^{-1}\cdots\omega_{n-1}^{-1})_*\nu)(\eta)
\geq{\rm d}\nu(\eta){\rm e}^{-n(\beta+\frac{\delta}{2})}.
\end{eqnarray*}
Thus
\begin{eqnarray*}
\nu\{\omega_{n-1}\cdots\omega_0\eta|\eta\in\Sigma_{([\omega],\xi)}^n\}
\geq\frac12{\rm e}^{-n(\beta+\frac{\delta}{2})}\geq {\rm e}^{-\frac{n\delta}{2}}{\rm e}^{-n(\beta+\frac{\delta}{2})}=
{\rm e}^{-n(\beta+\delta)}.
\end{eqnarray*}
By the definitions of $\Sigma_{([\omega],\xi)}^n$ and $\tilde{\Sigma}_n$, $([\omega],\xi)\in\tilde{\Sigma}_n.$
By the arbitrariness of $([\omega],\xi)$ we have $\hat{\Sigma}_n\subset\tilde{\Sigma}_n$ and therefore
the set $\tilde{\Sigma}_n$ has measure at least $1-\delta$ for all $n$  large enough.
The lemma follows since $\mu\ast\nu=\nu$ and the fact that the set $\tilde{\Sigma}_n$ has measure at least $1-\delta$ with respect to $\mu^{\otimes\mathbb{N}}\times\nu$ for $n$ sufficiently large
implies $\Sigma_n$ has measure at least $1-\delta$ with respect to $\nu$.
\end{proof}

\begin{theorem}
 $\beta(\widetilde{M}(\infty),\nu)=\beta\geq \beta(\mu)$.
\end{theorem}

\begin{proof}
Let $D$ and $\Sigma_n$ be as Lemma \ref{sigman} and fix $\delta>0$ small enough.
Let $\{\xi_i\in\Sigma_n|i\in I_n\}$ be a set of points such that $\{B(\xi_i,{\rm e}^{-n(A-\delta)})\}_{i\in I_n}$
is a cover of $\Sigma_n.$ Without loss of generality, we assume that for $i\neq j$, $\xi_j\notin B(\xi_i,{\rm e}^{-n(A-\delta)})$.
Thus $d_{p,D}(\xi_i,\xi_j)>{\rm e}^{-n(A-\delta)}$, that is $t_{p,D}(\xi_i,\xi_j)<n(A-\delta)$, for $i\neq j$.
Let $q_i=\sigma_{p,\xi_i}(n(A-\delta)),i\in I_n,$ then by the definition of $t_{p,D}$, we have
$d(q_i,q_j)>D$ for all $i\neq j.$
Hence the index of the cover has an upper bound $L$ independently of $n$, that is,
\begin{equation}
|I_n|\leq L{e}^{n(\beta+\delta)}.
\end{equation}
In fact, by the definition of $\Sigma_n$, we have $\nu(B(\xi_i,{\rm e}^{-n(A-\delta)}))\geq {e}^{-n(\beta+\delta)},$
that is
$$1\leq {e}^{n(\beta+\delta)}\nu(B(\xi_i,{\rm e}^{-n(A-\delta)})),$$
  we obtained that
  $$|I_n|=\sum_{i\in I_n}1\leq{e}^{n(\beta+\delta)}\sum_{i\in I_n}\nu(B(\xi_i,{\rm e}^{-n(A-\delta)}))\leq L{e}^{n(\beta+\delta)}$$
since the union of the balls has measure at most 1 and $dim\widetilde{M}(\infty)<\infty$.

Let $\Omega_n$ be a subset of $\Gamma^{\mathbb{N}}$ such that for any $[\omega]\in\Omega_n$:
\begin{itemize}
  \item (I) $[\omega](\infty)=\lim_{m\rightarrow\infty}[\omega]^{(m)}p\in \Sigma_n$;
  \item (II) $d([\omega]^{(n)}p,c_{p,[\omega](\infty)})\leq R$;
  \item (III) $\left|\frac1nd([\omega]^{(n)}p,p)-A\right|\leq \delta$;
  \item (IV) $-\frac{1}{n}\ln\mu^{n}([\omega]^{(n)})\geq \beta(\mu)-\delta.$
\end{itemize}
From Lemma \ref{sigman}, Corollary \ref{co27}, Lemma \ref{Guivarch} and \cite{KV}, we can obtain that the measure of
$\Omega_n$ is at least $1-2\delta$ if $n$ and $R$  large enough.
In the following we always assume that $n$ and $R$ be chosen such that $\Omega_n$ has measure at least $1-2\delta$ with respect to $\mu^{\otimes\mathbb{N}}$.

Since $\{B(\xi_i,{\rm e}^{-n(A-\delta)})\}_{i\in I_n}$ is a cover of $\Sigma_n$,
we get
\begin{eqnarray}\label{1-2delta}
1-2\delta&\leq&\mu^{\otimes\mathbb{N}}\{[\omega]\in\Omega_n|[\omega](\infty)\in\Sigma_n\}\nonumber\\
&\leq&\sum_{i\in I_n}\mu^{\otimes\mathbb{N}}\{[\omega]\in\Omega_n|[\omega](\infty)\in B(\xi_i,{\rm e}^{-n(A-\delta)})\}.
\end{eqnarray}
If $[\omega]\in\Omega_n$ such that $[\omega](\infty)\in B(\xi_i,{\rm e}^{-n(A-\delta)})$, then $d_{p,D}(\xi_i,[\omega](\infty))\leq{\rm e}^{-n(A-\delta)},$
that is, $t_{p,D}(\xi_i,[\omega](\infty))\geq n(A-\delta),$ thus
$$d(c_{p,\xi_i}(n(A-\delta)),c_{p,[\omega](\infty)}(n(A-\delta)))\leq D.$$
Choose $t$ such that $d(c_{p,[\omega](\infty)}(t),[\omega]^{(n)}p)=d(c_{p,[\omega](\infty)},[\omega]^{(n)}p)$, then we have
\begin{eqnarray}\label{dgammanpsigmap}
&&d([\omega]^{(n)}p,c_{p,\xi_i}(n(A-\delta)))\nonumber\\
&\leq&d([\omega]^{(n)}p,c_{p,[\omega](\infty)}(t))\nonumber\\
&&~~+d(c_{p,[\omega](\infty)}(t),c_{p,[\omega](\infty)}(n(A-\delta)))\nonumber\\
&&~~+d(c_{p,[\omega](\infty)}(n(A-\delta)),c_{p,\xi_i}(n(A-\delta)))\nonumber\\
&\leq&R+|t-n(A-\delta)|+D.
\end{eqnarray}

\begin{figure}[hb]
		\centering
		\includegraphics[width=0.8\textwidth]{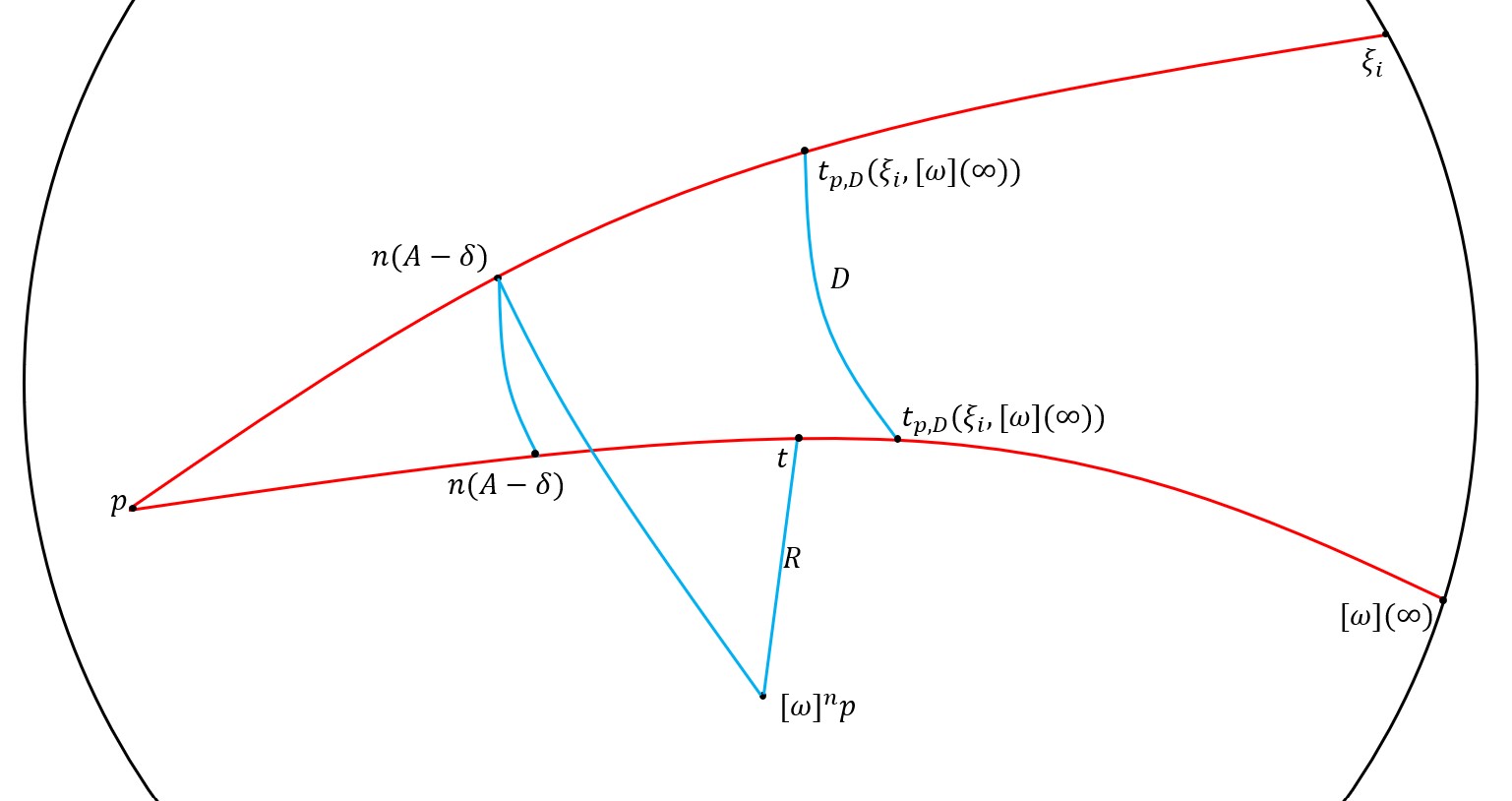}
\end{figure}

 Since  by (III),
 \begin{eqnarray*}
 t-n(A-\delta)&\leq &d([\omega]^{(n)}p,p)+R-n(A-\delta)\\&\leq& R+nA+n\delta-n(A-\delta)=R+2n\delta,\\
 t-n(A-\delta)&\geq& d([\omega]^{(n)}p,p)-R-n(A-\delta)\\&\geq &-R+nA-n\delta-n(A-\delta)=-R.
 \end{eqnarray*}
Thus by (\ref{dgammanpsigmap}) we obtain
\begin{eqnarray}
d([\omega]^{(n)}p,c_{p,\xi_i}(n(A-\delta)))&\leq&R+|t-n(A-\delta)|+D\nonumber\\
&\leq&R+R+2n\delta+D=2R+D+2n\delta.
\end{eqnarray}
The inequality shows that
$$\sharp\left\{[\omega]\in \Gamma^{\mathbb{N}}\left|[\omega]^{(n)}p\in B(c_{p,\xi_i}(n(A-\delta)),2R+D+2n\delta)\right.\right\}\leq C{\rm e}^{C(2R+D+2n\delta)},$$
where $C$ is a constant independent of $n$ and depends on $\widetilde{M}$ and $\Gamma$.
By (IV) we have
\begin{equation}\label{}
  \mu^{n}([\omega]^{(n)})\leq {\rm e}^{-n(\beta(\mu)-\delta)}.
\end{equation}
Therefore, we get
\begin{eqnarray}
   1-2\delta&\leq&|I_n|C{\rm e}^{C(2R+D+2n\delta)} {\rm e}^{-n(\beta(\mu)-\delta)}\nonumber \\
   &\leq& L{e}^{n(\beta+\delta)}C{\rm e}^{C(2R+D+2n\delta)} {\rm e}^{-n(\beta(\mu)-\delta)}\nonumber  \\
 &=&   LC{e}^{n(\beta-\beta(\mu)+2\delta+2\delta C)}{\rm e}^{C(2R+D)} .
\end{eqnarray}
Since $n$ is sufficiently large and $\delta>0$ is arbitrary, we conclude that $\beta\geq\beta(\mu).$
\end{proof}

\section*{\textbf{Acknowledgements}}
Fei Liu is partially supported by 
Natural Science Foundation of Shandong Province under Grant No.~ZR2020MA017 and No.~ZR2021MA064.
Yinghan Zhang is partially supported by Natural Science Foundation of China under Grant No.~12101370
and Natural Science Foundation of Shandong Province under Grant ZR2020QA020.

\end{document}